\documentclass[12pt]{amsart}
\usepackage{tikz}

\usepackage{amsmath,amsfonts,amsthm,amssymb,amscd, verbatim, graphicx,textcomp, hyperref}
\makeatletter
\newcommand{\xRightarrow}[2][]{\ext@arrow 0359\Rightarrowfill@{#1}{#2}}
\makeatother
\usepackage[notcite, notref]{}
\usepackage{upgreek}
\usepackage{pdflscape}
\usepackage[margin=2.5cm]{geometry}
\usepackage{pinlabel}
\usepackage{pdflscape}
\usepackage{tabularx}
\usepackage{latexsym}
\usepackage{hyperref}
\usepackage{multirow}
\usepackage{float}
\usepackage{xypic}
\usepackage{mathtools}

\usepackage{color}
\usepackage{stmaryrd}

\usepackage{makecell}

\newtheorem{Thm}{Theorem}[section]
\newtheorem{Cor}[Thm]{Corollary}
\newtheorem{Conj}[Thm]{Conjecture}
\newtheorem{Prop}[Thm]{Proposition}

\newtheorem{Lem}[Thm]{Lemma}

\theoremstyle{definition}

\newtheorem{example}[Thm]{Example}

\theoremstyle{remark}

\numberwithin{equation}{section}

\newcommand{\Aut}{\operatorname{Aut}}

\newcommand{\stab}{\operatorname{stab}}

\newcommand{\Out}{\operatorname{Out}}

\renewcommand{\dim}{\operatorname{dim}}

\newcommand{\ord}{\operatorname{ord}}

\newcommand{\Stab}{\operatorname{Stab}}

\newcommand{\D}{\mathcal{D}}

\newcommand{\m}{\textbf{m}}

\newcommand{\w}{\textbf{w}}

\newcommand{\Irr}{\operatorname{Irr}}

\newcommand{\Spin}{\operatorname{Spin}}

\newcommand{\Sol}{\operatorname{Sol}}

\newcommand{\GL}{\operatorname{GL}}
\newcommand{\SL}{\operatorname{SL}}

\newcommand{\PSp}{\operatorname{PSp}}

\newcommand{\DI}{\operatorname{DI}}

\renewcommand{\epsilon}{\varepsilon}
\renewcommand{\bar}{\overline}
\renewcommand{\hat}{\widehat}
\renewcommand{\leq}{\leqslant}

\newcommand{\K}{\mathcal{K} }

\newcommand{\N}{\mathcal{N}}
\newcommand{\F}{\mathcal{F}}

\newcommand{\C}{\mathbf{C}}

\newcommand{\A}{\mathbf{A}}

\newcommand{\E}{\mathcal{E}}
\newcommand{\W}{\mathcal{W}}

\renewcommand{\H}{\mathcal{H}}

\renewcommand{\O}{\mathcal{O}}

\renewcommand{\k}{\textbf{k}}

\begin{document}

\title{A $2$-compact group as a spets}




\author{Jason Semeraro}
\address{Heilbronn Institute for Mathematical Research, Department of 
Mathematics, University of Leicester,  United Kingdom}
\email{jpgs1@leicester.ac.uk}

\begin{abstract} 
In 1993, Brou\'{e}, Malle and Michel initiated the study of \textit{spetses} on the Greek island bearing the same name. These are mysterious objects attached to non-real Weyl groups. In algebraic topology, a $p$-compact group $\mathbf{X}$ is a space which is a homotopy-theoretic $p$-local analogue of a compact Lie group. A connected $p$-compact group $\mathbf{X}$ is determined by its root datum which in turn determines its Weyl group $W_\mathbf{X}$. In this article we give strong numerical evidence for a connection between these two objects
 by considering the case when $\mathbf{X}$ is the exotic $2$-compact group $\DI(4)$ constructed by Dwyer--Wilkerson and $W_\mathbf{X}$ is the complex reflection group $G_{24} \cong \GL_3(2) \times C_2$. Inspired by results in Deligne--Lusztig theory for classical groups, if $q$ is an odd prime power we propose a set $\Irr(\mathbf{X}(q))$ of `ordinary irreducible characters' associated to the space $\mathbf{X}(q)$ of homotopy fixed points under the unstable Adams operation $\psi^q$. Notably $\Irr(\mathbf{X}(q))$ includes the set of unipotent characters associated to $G_{24}$ constructed by Brou\'{e}, Malle and Michel from the Hecke algebra of $G_{24}$ using the theory of spetses. By regarding $\mathbf{X}(q)$ as the classifying space of a Benson--Solomon fusion system $\Sol(q)$ we formulate and prove an analogue of Robinson's ordinary weight conjecture that the number of characters of defect $d$ in $\Irr(\mathbf{X}(q))$ can be counted locally. 


\end{abstract}


\subjclass[2010]{}

\maketitle

\section{Introduction}
Let $k$ be an algebraically closed field of characteristic $p$. In \cite[Section 4]{kessar2018weight}, the authors introduce the concept of an \textit{$\F$-compatible family} associated to a saturated fusion system $\F$ on a finite $p$-group $S$. This is a family $\alpha = (\alpha_Q)_{Q \in \F^c}$ of cohomology classes $\alpha_Q \in H^2(\Out_\F(Q),k^\times)$ satisfying certain compatibility conditions, where $\F^c$ denotes the set of all $\F$-centric subgroups. The motivation for considering such families comes from block theory where, if $B$ is a block of $kG$ for some finite group $G$, there is a compatible family of K\"{u}lshammer--Puig classes $\alpha$ associated to the fusion system $\F$ of $B$ on its defect group. In this situation, we say that the pair $(\F,\alpha)$ \textit{realizes} $B$. Conjectures relating the local and global properties of $B$ involve $\alpha$ on the local side, and for this reason it becomes possible to make several conjectures for arbitrary pairs $(\F,\alpha)$ (see \cite[Section 2]{kessar2018weight}).

In a separate paper together with Lynd \cite{lynd2017weights}, the author shows that, when $\F$ is a Benson--Solomon fusion system $\Sol(q)$, any compatible family associated to $\F$ is trivial. Moreover the set of all $\F$-centric radical subgroups and their $\F$-outer automorphism groups is listed, and used to show that the number of weights associated to $\F$ is $12$, independently of $q$. Here we prove an analogous result (Theorem \ref{t:main}) for the \textit{ordinary} weights associated to $\Sol(q)$ and its close relative, the $2$-fusion system $\H$ of $\Spin_7(q)$. Precisely, we show that for each $d \ge 0$ the number of $2$-weights of defect $d$ can be expressed as a polynomial in the $2$-part of $q^2-1$. As a consequence, for the principal $2$-block $B$ of $\Spin_7(q)$ we prove Robinson's Ordinary Weight Conjecture that these numbers count characters of defect $d$ in $B$. 

There is no block with fusion system $\Sol(q)$ (see \cite{Kessar2006}, \cite[Theorem 9.34]{CravenTheory}), but it seems one can still construct a global object and conjecture a local-global correspondence in which $\Sol(q)$ features on the local side. Via the natural generalizations of certain results in Deligne--Lusztig theory it is possible to associate a set of irreducible characters to $\Sol(q)$ with a distinguished subset of unipotent characters being exactly those constructed using the  Brou\'{e}--Malle--Michel theory of spetses \cite{broue1999towards}.  In Section \ref{s:conclude} we state and prove an analogue of Robinson's weight conjecture that the number of such characters is given purely in terms of the fusion system. Finally, we verify several of the weight conjectures considered in \cite[Section 2]{kessar2018weight} for $\Sol(q)$.  


\section{Main results}\label{s:main}
Let $\F$ be a saturated fusion system on a finite $p$-group $S$, $\alpha$ be an $\F$-compatible family and $d \ge 0$ be a non-negative integer. Following \cite[Section 2]{kessar2018weight}, for any $\F$-centric subgroup $P$ of $S$ let $\N_P$
be the set of non-empty normal chains $\sigma$ of $p$-subgroups of 
$\Out_\F(P)$ starting  at the trivial subgroup; that is, chains of 
the form
$$\sigma = (1=X_0 < X_1 <\cdots < X_m)$$
with the property that $X_i$ is normal in $X_m$ for $0\leq i\leq m$. 
We set $|\sigma| = m$, and call $m$ the {\it length} of $\sigma$. We also define
$$\Irr^d(P):=\{\mu \in \Irr(P) \mid \mu(1)=p^{-d}|P|\},$$ 
the set of ordinary irreducible characters of $P$ of defect $d$ and 
define $\W^d_P = \N_P \times \Irr^d(P)$. The obvious actions of the group $\Out_\F(P)$ on 
$\N_P$ and $\Irr^d(P)$ yield an action on  $\W^d_P$. If $(\sigma, \mu) \in \mathcal{W}_P^d$, let $I(\sigma)$ and $I(\sigma,\mu)$ denote the stabilizers in $\Out_\F(P)$ of $\sigma$ and $(\sigma,\mu)$ respectively
under these actions. Now set
$$\w_P(\F,\alpha,d):=
\sum_{\sigma \in \N_P/\Out_\F(P)} (-1)^{|\sigma|} 
\sum_{\mu \in \Irr^d(P)/I(\sigma)} z(k_\alpha I(\sigma,\mu)),$$ where  $k_\alpha I(\sigma,\mu)$  is group algebra of $I(\sigma,\mu)$ twisted with respect to $\alpha$ and $z(\cdot)$ denotes the number of projective simple modules. Finally, we set
$$ \m(\F,\alpha,d):=\sum_{P \in \F^{c}/\F} \w_P(\F,\alpha,d) \hspace{4mm} \mbox{ and } \hspace{4mm} \m(\F,\alpha):=\sum_{d \ge 0} \m(\F,\alpha,d). 
$$

By \cite[Lemma 7.5]{kessar2018weight}, the quantities $\m(\F,\alpha,d)$ and  $\m(\F,\alpha)$ remain unchanged on restricting the sums to isomorphism classes of $\F$-centric radical subgroups. If $d \ge 0$ is an integer and $B$ is a block of $kG$ for some finite group $G$, we write $\k_d(B)$ for the number of irreducible characters of $G$ of defect $d$ in the block $B$. The relevance of the invariant $\m(\F,\alpha,d)$ is the following:

\begin{Conj}[Robinson]\label{c:rob}
Let $\F$ be a saturated fusion system on a finite $p$-group $S$, and let $\alpha$ be an $\F$-compatible family. Suppose that $(\F,\alpha)$ realizes a block $B$ of $kG$ for some finite group $G$. Then $$\normalfont{\m}(\F,\alpha,d)=\k_d(B).$$
\end{Conj}
We fix some notation which will be used until Section \ref{s:conclude}. Let $q$ be a (fixed) odd prime power, define $l:=v_2(q^2-1)-3$ and set $x:=2^l$, where $v_p$ is the $p$-adic valuation. Notice that $l \ge 0$ since $8 \mid q^2-1$. Moreover set:
\begin{itemize}
\item[(1)] $H$ equal to $\Spin_7(q)$ and $S$ equal to a (fixed) Sylow $2$-subgroup of $H$;
\item[(2)] $\H = \H(q) = \F_S(H)$ equal to the $2$-fusion system of $H$ on $S$;
\item[(3)] $\F=\F(q):=\Sol(q)$, a Benson--Solomon fusion system on $S$ which contains $\H$ (see \cite{lynd2017weights}.)
\end{itemize}


In \cite[Theorem 1.1]{lynd2017weights} it was shown that there are no non-trivial $\F$-compatible families associated to $\F$. Here is our first main result:

\begin{Thm}\label{t:main}
Let $\D \in \{\H,\F\}$ and $d \ge 0$. $\normalfont{\m}(\D,0,d)$ is expressible as a rational polynomial in $x$. Moreover its precise values are listed in Table \ref{table:mdod} when  $\normalfont{\m}(\D,0,d)$ is (possibly) non-zero. \end{Thm}

\begin{table}[]
\renewcommand{\arraystretch}{1.4}
\centering
\caption{Possibly non-zero values of $\m(\D,0,d)$ for $\D \in \{\H,\F\}$ }
\label{table:mdod}
\begin{tabular}{|c|c|c|}
\hline
$d$       & $\m(\H,0,d)$                                                             & $\m(\F,0,d)$ \\ \hline
$4$     & $2$                                                                  &  $2$           \\ \hline
$l+5$   & $2 (x-1)$                                                    &  $0$           \\ \hline
$l+6$    & $2x+7$                                                          &  $2x-1$           \\ \hline
$l+7$    & $4$                                                                  &  $4$           \\ \hline
$2l+5$   & $3x(x-1)$                                         &  $x(x-1)$           \\ \hline
$2l+6$    & $12x$                                                    &  $4x$           \\ \hline
$3l+6$    & $\frac{4}{3}  x^3-3x^2+\frac{11}{3} x-1$ &  $\frac{4}{21} x^3-x^2+\frac{11}{3} x-\frac{20}{7}$           \\ \hline
$3l+7$     & $12x^2-8x+2$                                         &  $4x^2-6x+4$           \\ \hline
$3l+8$    & $14x-4$                                                  &  $6x-4$           \\ \hline
$3l+9$    & $8x+4$                                                          &  $8x+4$           \\ \hline
$3l+10$   & $16$                                                                 &  $16$            \\ \hline \hline
$\m(\D,0)$ & $\frac{4}{3} x^3+12x^2+\frac{92}{3}x+28$ & $\frac{4}{21} x^3+4x^2+\frac{50}{3}x+\frac{155}{7}$ \\ \hline
\end{tabular}
\end{table}

As a byproduct of our computations, we obtain:

\begin{Thm}\label{t:owcforspin}
For all odd prime powers $q$, Conjecture \ref{c:rob} holds for the principal $2$-block of $\Spin_7(q)$.
\end{Thm}

We now explain why Theorem \ref{t:main} is surprising when $\D=\F.$ In \cite{lynd2017weights}, the local structure of $\D$ is determined by treating the case $l = 0$ separately from the generic case $l > 0$ because the structure of the centric radical subgroups and their automorphism groups does not admit a uniform treatment (see \cite[Tables 1,4] {lynd2017weights}.)  As a result we must handle the computation of $\m(\D,0,d)$ in the case $l=0$ separately. On the other hand when a polynomial we obtain for $l > 0$ is specialised to the case $l=0$ we always get the correct answer! When $\D=\H$, this phenomenon is explained by the fact that Theorem \ref{t:owcforspin} equates $\m(\D,0,d)$ with a polynomial obtained using character counting methods in Deligne--Lusztig theory (see Section \ref{s:conjectures}) which do not distingush between the cases $l=0$ and $l>0$. 

Now $\F$ is also the fusion system of a $2$-local finite group with classifying space $B\mathbf{X}(q)$, where $\mathbf{X}=\DI(4)$ is the space constructed by Dwyer--Wilkerson \cite{DwyerWilkerson1993} and $\mathbf{X}(q)$ denotes the space of homotopy fixed points under the unstable Adams operation $\psi^q$ acting on $\mathbf{X}$ (see Section \ref{ss:pcomp}). For this reason, we propose that the genericity of the polynomials in Table \ref{table:mdod} can be explained by regarding $\DI(4)$ as a \textit{spets} associated to the non-real Weyl group $G_{24}$. Thus we define a set $\Irr^u(\mathbf{X}(q))$ of unipotent characters to be those given in \cite{broue2014split} for $G_{24}$. What about the remaining characters? 
If $B$ denotes the principal $2$-block of $H$, a result of Cabanes--Enguehard in Deligne--Lusztig theory (Proposition \ref{p:cabanes}) states that  $\Irr(B)$ is a union $\bigcup_s \E(H,s)$ of \textit{Lusztig series} taken over conjugacy class representatives of $2$-elements in the Langlands dual group $H^*:=\Aut(\PSp_6(q))$. Here, $\E(H,s)$ is in bijection with the set of unipotent characters of the centralizer $C_{H^*}(s)$ (see the discussion which precedes Proposition \ref{p:cabanes}.)  Treating $\mathbf{X}(q)$ like a Langlands self-dual finite group of Lie-type, and using the fact that centralizers of non-trivial $2$-elements in $\F$ are contained in $\H$ we are led to define \begin{equation}\label{e:fls} \E(\mathbf{X}(q),s):=\begin{cases} \Irr^u(\mathbf{X}(q)) & \mbox{ if $s=1$;} \\ \E(H,s) & \mbox{ otherwise, }
 \end{cases} \mbox{ and } \Irr(\mathbf{X}(q))=\bigcup_s \E(\mathbf{X}(q),s),\end{equation} where $s$ runs over a complete set of (fully $\F$-centralized) $\F$-conjugacy class representatives. We then set $\k_d(\mathbf{X}(q))$ equal to the  number of characters of defect $d$ in $\Irr(\mathbf{X}(q))$ and prove that the following analogue of Robinson's Conjecture \ref{c:rob} holds for the space $\mathbf{X}(q)$:

\begin{Thm}\label{t:exowc}
For each $d > 0$, we have $\normalfont{\m}(\F(q),0,d)=\normalfont{\k}_d(\mathbf{X}(q))$.
\end{Thm}

We prove Theorem \ref{t:exowc} in Section \ref{ss:spetses}. Note that the restriction on $d$ is necessary. Indeed six of the twenty-two spetsial characters in $\Irr^u(\mathbf{X}(q))$ have defect $0$ (possibly corresponding to `other blocks' associated to $\mathbf{X}(q)$). 



Finally, we use Theorem \ref{t:main} to give  further evidence towards the various conjectures considered in \cite[Section 2]{kessar2018weight}. Recall the definition of the invariant $\k(\D,\alpha)$ as given in \cite[Section 1]{kessar2018weight}. One of the main results in \cite{kessar2018weight} is that Alperin's Weight Conjecture implies $\k(\D,\alpha)=\m(\D,\alpha)$.

\begin{Thm}\label{t:conjectures}
Let $\D \in \{\H,\F\}$ and assume that  $\normalfont{\k}(\D,0)=\normalfont{\m}(\D,0)$. Then \cite[Conjectures 2.1, 2.5, 2.8, 2.9 and 2.10]{kessar2018weight} are true for the pair $(\D,0)$. 
\end{Thm}

Since this work was first made publicly available in June 2019, the author, in joint work with Radha Kessar and Gunter Malle, has greatly developed the connection exposed here between $p$-compact groups and spetses (see \cite{KMS20}). In light of this, Theorem  \ref{t:exowc} implies the conclusion of the Ordinary Weight Conjecture \cite[Conjecture 1]{KMS20} holds for the principal $2$-block of the $\mathbb{Z}_2$-spets associated to $G_{24}$ when $d \notin \{4,l+7\}$.

\subsection*{Structure of the paper} 
For $\D \in \{\H,\F\}$ in Section \ref{s:centrads} we recall the classification of $\mathcal{D}$-centric radical subgroups from \cite{lynd2017weights},  emphasising the provision of computable descriptions of these groups. Using this classification, in Section \ref{s:charsofdcr} we invoke elementary character theory to prove that when $l >0$ the quantity $\w_P(\D,0,d)$ is a rational polynomial in $x$ for each $d \ge 0$ (Corollary \ref{c:rat}). We explicitly calculate these polynomials using a computer (Theorem \ref{t:table4}) and deduce Theorem \ref{t:main}. In Section \ref{s:conjectures} we prove Theorems \ref{t:owcforspin} and \ref{t:conjectures}.  In Section \ref{s:conclude} we briefly introduce spetses and $p$-compact groups before proving Theorem \ref{t:exowc}.

\subsection*{Acknowledgements} 

The author would like to thank  Radha Kessar, Jesper Grodal and Gunter Malle for their comments on earlier versions of this manuscript. Thanks also to Frank L\"{u}beck for providing Tables \ref{table:1mod4} and \ref{table:3mod4} and to Jay Taylor for helping me understand them. I am also grateful to David Craven for suggesting the main approach used in the proof of Theorem \ref{t:main}, and to Markus Linckelmann and Justin Lynd  for helpful conversations. Finally I would like to thank the Mathematisches Forschungsinstitut Oberwolfach for its hospitality during the week-long workshop ``Representations of Finite Groups'' in March 2019. It was there that many of the ideas in this paper were conceived. Finally I would like to thank the anonymous referees for their careful reading and suggestions which have led to numerous improvements.

\section{Centric radical subgroups and their automorphisms}\label{s:centrads}
For $\D \in \{\H,\F\}$ a complete classification of the set $\D^{cr}$ of $\D$-centric radical subgroups of $S$ and their automorphism groups is one of the main results in \cite{lynd2017weights}. Here we list these subgroups in the first column of Table \ref{table:wpdod}. Broadly, this classification is understood using a certain group $K$ which contains $S$ as a Sylow $2$-subgroup, representing the `difference' between $\H$ and $\F$ in the sense that $\F=\langle \H,\K \rangle_S$ is the smallest fusion system on $S$ generated by morphisms in $\H$ and $\K=\F_S(K)$. Utilising the precise results contained in \cite[Sections 2 and 3]{lynd2017weights}, in this section we describe the elements of $\D^{cr}$ and their automorphism groups.  In particular, we note the following:

\begin{Prop}\label{p:centrads}
Let $\D \in \{\H,\F\}$ and $P \in \D^{cr}$ be a $\D$-centric radical subgroup. Then either $\Out_\D(P) \le \Out_\K(P)$ (with index at most $3$) or else $P \in \{R_{1^7},R'_{1^7},R_{1^52},C_S(E/Z),A,C_S(E)\}$.
\end{Prop}

We begin by considering those $P \in \D^{cr}$ for which $\Out_\D(P) \le \Out_\K(P)$. We work with an explicit $6$-dimensional representation of the group $K$ which we now describe. Under the natural inclusion of $\SL_2(q)$ into $\SL_2(q^2)$, let $N:=N_{\SL_2(q^2)}(\SL_2(q))$ be its normalizer. Form the wreath product $W:=N \wr S_3$ and let $N_0:=N_1 \times N_2 \times N_3$ ($N_i$ isomorphic with $N$) and $X=S_3=\langle \tau,\gamma \rangle $ be the base and acting group, where $\tau$ and $\gamma$ act like $(1,2)$ and $(2,3)$ respectively. The natural representations $\rho_i: N_i \hookrightarrow \SL_2(q^2)$ induce a representation $\rho$ of $W$ given by $$N_0 \mapsto \left(\begin{matrix} 
\rho_1(N_1) & 0 & 0 \\
0 & \rho_2(N_2) & 0 \\
0 & 0 & \rho_3(N_3)  \\
\end{matrix}\right), \hspace{2mm} \tau \mapsto \left(\begin{matrix} 
0 & I_2 & 0 \\
I_2 & 0 & 0 \\
0 & 0 & I_2  \\
\end{matrix}\right), \hspace{2mm} \gamma \mapsto \left(\begin{matrix} 
I_2 & 0 & 0 \\
0 & 0 & I_2 \\
0 & I_2 & 0  \\
\end{matrix}\right),$$ where $I_2$ denotes the $2 \times 2$ identity matrix. This leads to a representation of $K:=\hat{K}/Z(\hat{K})$, where $\hat{K}:=O^2(N_0)C_{N_0}(X)X \le W$ and $Z(\hat{K})=\langle (-1,-1,-1)\rangle \le N$. 
Now $O^2(N_0)=\hat{L}_1 \times \hat{L}_2 \times \hat{L}_3$ is a direct product of three copies of $\SL_2(q)$ permuted transitively by $X$. Write $L_i$ for the image in $K$ of $\hat{L}_i$ and let $R_i \cong Q_{2^{l+3}}$ be a Sylow $2$-subgroup of $L_i$ for $i=1,2,3$. As in \cite[Notation 2.12]{lynd2017weights} when $l>0$ there are $R_i$-conjugacy classes of subgroups isomorphic to $Q_8$ with representatives $Q_i$ and $Q_i'$ chosen so that $X$ transitively permutes $\{Q_1,Q_2,Q_3\}$ and $\{Q'_1,Q'_2,Q'_3\}$. These classes are fused by an element $\textbf{c}$ whose preimage (in $\hat{K}$) lies in $C_{N_0}(X)$. We may multiply $\textbf{c}$ on the left by a certain diagonal element in $R_1R_2R_3$ to produce an element $\textbf{d}$ which commutes with $\tau$, so that $S=R_1R_2R_3\langle \textbf{d},\tau \rangle$. We also set $\tau':=\textbf{d}\tau$  (see \cite[Notation 2.12(f),(g)]{lynd2017weights}.) For $\D\in \{\H,\F\},$ the above discussion describes the $\D$-centric radical subgroups listed in the first $14$ and $11$ rows of \cite[Tables 3 and 4]{lynd2017weights} respectively (note that the group $C_S(U) \in \F^{cr}$ admits the alternative description $R_1R_2R_3\langle \textbf{d} \rangle$). 

We now turn to the remaining groups in Proposition \ref{p:centrads}. We fix an embedding $G_{24} \hookrightarrow \GL_3(\mathbb{Z}_2)$ of the $\mathbb{Z}_2$-reflection group $G_{24}$ from which we obtain representations $$\varphi_k:G_{24} \hookrightarrow \GL_3(\mathbb{Z}/2^k), \hspace{3mm} \mbox{ for each } k \ge 2$$  (see  \cite[Theorem 4.1]{DwyerWilkerson1993}). We use this family of representations to identify $S$ with a Sylow $2$-subgroup of $T \rtimes G_{24}$ where $T$ is the torus of $\mathcal{F}$: a homocyclic subgroup of rank $3$ and exponent $l+2$ (see \cite[Section 2.5]{lynd2017weights}). As in \cite[Section 2.6]{lynd2017weights}, set $Z:=Z(S)$, $U:=Z(L_1L_2L_3)$,  $E:=\Omega_1(T)$ and $A=T\langle \textbf{d} \rangle$. Then $Z < U < E < A$ is a chain of elementary abelian subgroups. The groups $C_S(E)=T\langle \textbf{d}\rangle$ and $C_S(E/Z)$ both contain $T$ and are $\F$-centric radical; the latter group is in addition $\H$-centric radical. Note that $C_S(U)/C_S(E)$ and $C_S(E/Z)/C_S(E)$ are unipotent radical subgroups in  $O^2(G_{24}) \cong \GL_3(2) \cong \Out_\F(C_S(E))$.
We also have $\Out_\F(A) \cong \GL_4(2)$, constructed from the unique non-split extension of $\GL_4(2)$ by $\mathbb{F}_2^4$ (see \cite[Theorem 9.15]{CravenTheory}.)  The remaining elements of $\D^{cr}$ are described in \cite[Proposition 3.2]{lynd2017weights} and labelled $P=R_{1^7}$, $P=R'_{1^7}$ and $P=R_{1^52}$. The first two of these have orders $2^7$ and $2^6$ respectively and can be constructed inside $\Spin_7(3)$ or $\Spin_7(9)$. The group $P=R_{1^52}$ is isomorphic to $(Y_1 \times Y_2 \times Y_3)/Z$ where $Y_1=Q_8, Y_2= D_8$, $Y_3= Q_{2^{l+3}}$ and $Z:=\langle (z_1,z_2,1),(z_1,1,z_3)\rangle$ with $\langle z_i \rangle=Z(Y_i)$ for $i=1,2,3$. Hence $P$ is a central product $2_{-}^{1+4} * Q_{2^{l+3}}$, and  $\Out_\D(P) = \Out(2_{-}^{1+4})= GO_4^-(2) \cong S_5$ acts trivially on the image of $Y_3$ in $P$. 

Finally, to each $P \in \D^{cr}$ we associate an integer $0 \le a \le 3$ called its \textit{type} listed in the third column of Table \ref{table:wpdod}. In all cases we have $v_2(|P|)=al+b$ for some integer $b \le 10$, and thus the possible defects of characters of $P$ are understood in terms of $a$.

\section{Characters of centric radical subgroups}\label{s:charsofdcr}
One useful observation is that the elements of $\D^{cr}$ satisfying $\Out_\D(P) \le \Out_\K(P)$ are groups containing a normal subgroup isomorphic to a central product of three generalised quaternion groups of index at most $4$. The characters of such groups have an elementary description by virtue of the following result. If $G_0 \le G$ are groups, we let $\Aut(G,G_0)$ denote the set of automorphisms of $G$ which fix $G_0$. 

\begin{Lem}[Method of little groups]\label{l:cliff}
Let $G$ be a finite group and $G_0$ be a normal subgroup of $G$, and define $\Theta:=[\Irr(G_0)/G]$ to be set of orbit representatives for the action of $G$ on $\Irr(G_0)$. Suppose that each $\theta \in \Irr(G_0)$ extends to a character $\hat{\theta} \in \Irr(I_G(\theta))$ (so $\hat{\theta}|_{G_0}=\theta$) where $I_G(\theta)$ denotes the inertia subgroup. The following hold:

\begin{itemize}
\item[(1)] There is a bijection $$\Phi: \{(\theta, \beta) \mid \theta \in \Theta \mbox{ and } \beta \in \Irr(I_G(\theta)/G_0)\} \longrightarrow \Irr(G)$$  which sends a pair $(\theta,\beta)$ to the character $(\beta \widehat{\theta})^G$, where $\widehat{\theta} \in \Irr(I_G(\theta))$ is some (any) character which extends $\theta$. 
\item[(2)] $\Phi$ is $\Aut(G,G_0)$-equivariant in the sense that for each pair $(\theta, \beta)$ as above and element $\alpha \in \Aut(G,G_0)$ the following hold:
\begin{itemize}
\item[(i)] $I_G(\theta)\alpha=I_G(\theta^\alpha)$;
\item[(ii)]  there exists an extension $\widehat{\theta^\alpha}:=\widehat{\theta}^\alpha \in \Irr(I_G(\theta^\alpha))$ of $\theta^\alpha$; and
\item[(iii)] we have, $\Phi((\theta^\alpha,\beta^\alpha))=\Phi(\theta,\beta)^\alpha.$
\end{itemize}
 \end{itemize}
\end{Lem}

\begin{proof}
(1) is proven in \cite[Theorem 4.2]{lynd2017weights} and follows from \cite[11.5]{CurtisReiner1990}. If $(\theta,\beta)$ and $\alpha$ are as in (2), then for each $g \in I_G(\theta)$ we have $\theta^{\alpha \circ c_{g\alpha}}=\theta^{c_g \circ \alpha}=\theta^\alpha$ from which we conclude that $g\alpha \in I_G(\theta^\alpha)$ and (i) holds. This shows that $\hat{\theta}^\alpha \in \Irr_G(I_G(\theta^\alpha))$ extends $\theta^\alpha$ and (ii) also holds. Finally, if $T = [G/I_G(\theta)]$ is a transversal then for each $g \in G$, $$(\hat{\theta}\beta)^G(g) \sum_{t \in T} \hat{\theta}((t\alpha)^{-1} g\alpha t \alpha)\beta((t\alpha)^{-1} g\alpha t \alpha) = \sum_{t \in T\alpha} \hat{\theta}(t^{-1} g\alpha t)\beta(t^{-1} g\alpha t) = (\hat{\theta}^\alpha \beta^\alpha)^G(g), $$ since $T\alpha$ is a transversal for $I_G(\theta^\alpha)$ in $G$ by (i), proving (iii). 
\end{proof}

When applying Lemma \ref{l:cliff} we will require a precise description of the irreducible characters of a generalized quaternion group. 



\begin{Lem}\label{l:genquat}
Let $R:=\langle a,b \mid a^{2^{l+2}}=b^4=1, a^{2^{l+1}}=b^2, b^{-1}ab=a^{-1} \rangle$ be a generalized quaternion group of order $2^{l+3}$ and let $\omega \in \mathbb{C}$ be a primitive root of unity of order $2^{l+2}$. The following hold:
\begin{itemize}
\item[(1)] $R$ has $2^{l+1}+3$ conjugacy classes with a complete set of representatives given by $\{b,ab\} \cup \{a^i \mid i=0,1,\ldots, 2^{l+1}\}$.
\item[(2)] Every non-linear irreducible character of $R$ is of degree $2$ and given by $$\psi_{u,t}(g):=\begin{cases} 0 & \mbox{ if $g \notin \langle a \rangle$ } \\ \omega^{si} + \omega^{-si} & \mbox{ if $g=a^i$, some $0 \le i \le 2^{l+2}-1$.} \end{cases}$$ where $s=2^{l-u}t$ for some $0 \le u \le l$ and $t=1,3,\ldots,2^{u+1}-1$.
\item[(3)] $R$ has four linear characters.
\end{itemize}
In particular, $|\{\chi \in \Irr(R) \mid Z(R) \nleq \ker(\chi)\}|=2^l$.
\end{Lem}

\begin{proof}
Part (1) is proved in \cite{lynd2017weights} and part (3) is immediate.  By \cite[Theorem 26.4]{james2001representations}, every character is either the lift of an irreducible character of $R/Z(R)$ or of the form $\psi^R$ where $\psi \in \Irr(\langle a \rangle)$ and $Z(R) \not \le \ker(\psi)$. Since $R/Z(R)$ is a dihedral group, and characters of $Q_8$ are well-known we can calculate the degree 2 characters inductively. Every character of $\langle a \rangle$ is given by $\varphi_t(a^i)=\omega^{ti}$ for some $0 \le t \le 2^{l+2}-1$, and one easily shows that:
\begin{itemize}
\item[(a)] $Z(R) \not \le \ker(\varphi_t) \Longleftrightarrow t \equiv 1 \mod 2$;
\item[(b)] $\varphi_t^ R = \varphi_s^ R  \Longleftrightarrow s+t \equiv 0 \mod 2^{l+2}$.
\end{itemize}
Thus, apart from the characters obtained by lifting $R/Z(R)$ we obtain $2^{l+1}-2^l=2^l$ distinct characters by inducing from $\langle a \rangle$ given by
$$\psi_{l,t}(g):=\begin{cases} 0 & \mbox{ if $g \notin \langle a \rangle$; } \\ \omega^{ti} + \omega^{-ti} & \mbox{ if $g=a^i$, some $0 \le i \le 2^{l+2}-1$,} \end{cases}$$ 
for each $t=1,3,\ldots,2^{l+1}-1$. (2) now follows by induction.
\end{proof}

The following example is illustrative of our general approach to calculating the integers $\w_P(\D,0,d)$. It is also included to emphasise that many of the computer computations in this paper replace lengthy and unenlightening by-hand calculations. Recall from the introduction that $x=2^l$.

\begin{example}\label{ex:csu}
Let $U$ be as in Section \ref{s:centrads}, so that $C_S(U) \in \F^{cr}$. For all $l > 0$ and $P=C_S(U)$ we have $$\w_P(\F,0,3l+6)=-\frac{1}{3}(8x^3-6x^2-8x+9).$$
\end{example}

\begin{proof}
By \cite[Notation 3.3]{lynd2017weights}, $P=(R_1 \times R_2 \times R_3)\langle \mathbf{c} \rangle/Z$ where $R_i \cong Q_{2^{l+3}}$ and $Z=\langle z_1z_2z_3\rangle$ with $\langle z_i \rangle = Z(R_i)$ for each $i$. Recall from Section \ref{s:centrads} that $\mathbf{c}$ interchanges the the two classes of subgroups isomorphic with $Q_8$ in each component and that $\Out_\F(P) \cong S_3$ transitively permutes the $R_i$. 

Note in particular that $v_2(|P|)=3l+9$, and since $d=3l+6$ we must consider characters of degree $8$.  Since every irreducible character $\theta$ of $R_1R_2R_3/Z$ extends to its inertia group $I_P(\theta)$, by Lemma \ref{l:cliff} for each $\chi \in \Irr^d(P)$, either:
\begin{itemize}
\item[(a)] $\chi = (\theta_1 \otimes \theta_2 \otimes \theta_3)^P$, for some $\theta_1 \otimes \theta_2 \otimes \theta_3 \in \Irr(R_1 \times R_2 \times R_3)$ of degree $4$ which is not fixed by $\mathbf{c}$ and whose kernel contains $Z$; or
\item[(b)] $\chi$ is one of two constituents of  $(\theta_1 \otimes \theta_2 \otimes \theta_3)^P$ for some $\theta_1 \otimes \theta_2 \otimes \theta_3 \in \Irr(R_1 \times R_2 \times R_3)$ of degree $8$ which is fixed by $\mathbf{c}$ and whose kernel contains $Z$.
\end{itemize}
From Lemma \ref{l:genquat} we see that $\mathbf{c}$ fixes all degree $2$ characters of the $R_i$ and interchanges two of the four linear characters in each component. Write $\Irr^d_{(a)}(P)$ and $\Irr^d_{(b)}(P)$ for the sets of characters in cases (a) and (b) respectively. By Lemma \ref{l:cliff}(2) the action of $\Out_\F(P) \cong S_3$ on $\Irr^{3l+6}(P)$ is described by permutation of the tensor factors in the corresponding character of $R_1 \times R_2 \times R_3$, and this action preserves the decomposition $\Irr^d(P)=\Irr^d_{(a)}(P)\sqcup \Irr^d_{(b)}(P)$ described above. Let $[c_\tau] \in \Out_\F(P)$ be the class of $c_\tau \in \Aut_\F(P)$ which acts by the permutation $(1,2)$ and set $\sigma_1:= (1)$ and $\sigma_2:=(1 < [c_\tau])$. Thus $\{\sigma_1,\sigma_2 \}$ is a complete set of $\Out_\F(P)$-representatives for elements of $\N_P$. We split the sum

$$\begin{array}{rcl} \w_P(\F,0,d) &=& \displaystyle\sum_{\sigma \in \N_P/\Out_\F(P)} (-1)^{|\sigma|} 
\sum_{\chi \in \Irr^d(P)/I(\sigma)} z(kI(\sigma,\chi)) \\ 
&=& \displaystyle\sum_{\chi \in \Irr^d(P)/I(\sigma_1)} z(kI(\sigma_1,\chi)) - \sum_{\chi \in \Irr^d(P)/I(\sigma_2)} z(kI(\sigma_2,\chi)), \\
\end{array}$$
and then for each $i=1,2$ we have $$\begin{array}{rcl} \displaystyle\sum_{\chi \in \Irr^d(P)/I(\sigma_i)} z(kI(\sigma_i,\chi))  &=& \displaystyle\sum_{\chi \in \Irr_{(a)}^d(P)/I(\sigma_i)} z(kI(\sigma_i,\chi)) + \displaystyle\sum_{\chi \in \Irr_{(b)}^d(P)/I(\sigma_i)} z(kI(\sigma_i,\chi)), \\
\end{array}$$
which gives a total of four sums to compute. The following notation will be helpful when enumerating characters whose kernel contains $Z$: 
$$\begin{array}{rcl}
\Irr_0(R_i) &:=& \{\mu \in \Irr(R_i), \mu(1)=2, Z(R_i) \le \ker(\mu)\}; \mbox{ and } \\ \Irr_1(R_i)&:=&\{\nu \in \Irr(R_i), \nu(1)=2, Z(R_i) \nleq \ker(\nu)\}.\end{array}$$ 
for each $1 \le i \le 3$. By Lemma \ref{l:genquat}, $|\Irr_0(R_i)|=x-1$ and $|\Irr_0(R_i)|=x$.
\newline \newline 
\underline{Case 1: $\chi \in \Irr_{(a)}^d(P)/I(\sigma_1)$.} Let $\rho$ be a linear character of $R_3$ which is moved by $\mathbf{c}$. In Table \ref{table:casea} we list orbit representatives $\theta \in \Irr(R_1 \times R_2 \times R_3)$ satisfying the conditions of (a) according to their `type', together with the number of such representatives, and their $I(\sigma_1)$-stabilizers. From the table, we obtain a (positive) contribution to $\w_P(\F,0,3l+6)$ of $${x-1 \choose 2} +{x \choose 2}$$ from characters of types 1 and 3. 
\newline \newline \underline{Case 2: $\chi \in \Irr_{(a)}^d(P)/I(\sigma_2)$}. We have $I(\sigma_2)=\langle [c_\tau] \rangle$ and from Table \ref{table:casea2} we see that the negative contribution to $\w_P(\F,0,3l+6)$ in this case is  $${x-1 \choose 2} +{x \choose 2} + (x-1)^2+x^2.$$ from characters of types 1, 3, 5   and 6.
\newline
\underline{Case 3: $\chi \in \Irr_{(b)}^d(P)/I(\sigma_2)$}. Table \ref{table:caseb} lists the possible orbit representatives satisfying (b), where we see that a negative contribution to $\w_P(\F,0,3l+6)$ of $$x(x-1)^2+2x^2(x-1)+(x-1)^2(x-2)$$ is provided by representatives of types 1, 3 and 5.
\newline
\newline
\underline{Case 4: $\chi \in \Irr_{(b)}^d(P)/I(\sigma_1)$.} Here we obtain a contribution of $$x(x-1)^2+\frac{(x-1)(x-2)(x-3)}{3}+2(x-1)$$ from characters of types 1, 4 and 5 (see Table \ref{table:caseb2}).

Combining Cases 1 to 4 we have $\w_P(\F,0,3l+6)$ is equal to $$\begin{array}{rcl}  && \displaystyle{- \left((x-1)^2+x^2 \right)+\frac{(x-1)(x-2)(x-3)}{3}+2(x-1) -2x^2(x-1)-(x-1)^2(x-2)}\\ &=& \displaystyle{-\frac{1}{3}(8x^3-6x^2-8x+9),} \end{array}$$ whence the result.

\begin{table}[]
\renewcommand{\arraystretch}{1.4}
\centering
\caption{$I(\sigma_1)$-orbits of characters in case (a) and their stabilizers  }
\label{table:casea}
\begin{tabular}{|c|c|c|c|c|}
\hline
Type & $\theta$                             & Conditions                                  & $\#$              & $\stab_{I(\sigma_1)}(\chi)$ \\ \hline
$1$  & $\mu_1 \otimes \mu_2 \otimes \rho$ & $\mu_i \in \Irr_0(R_i),$ $\mu_1 \neq \mu_2, $ & ${x-1 \choose 2}$ & $1$                       \\ \hline
$2$  & $\mu_1 \otimes \mu_2 \otimes \rho$ & $\mu_i \in \Irr_0(R_i),$ $\mu_1 = \mu_2, $    & $x-1$             & $\langle [c_\tau] \rangle$    \\ \hline
$3$  & $\nu_1 \otimes \nu_2 \otimes \rho$ & $\nu_i \in \Irr_1(R_i),$ $\nu_1 \neq \nu_2, $ & ${x \choose 2}$   & $1$                       \\ \hline
$4$  & $\nu_1 \otimes \nu_2 \otimes \rho$ & $\nu_i \in \Irr_1(R_i),$ $\nu_1 = \nu_2, $    & $x$               & $\langle [c_\tau] \rangle$    \\ \hline
\end{tabular}
\end{table}

\begin{table}[]
\renewcommand{\arraystretch}{1.4}
\centering
\caption{$I(\sigma_2)$-orbits of characters in case (a) and their stabilizers  }
\label{table:casea2}
\begin{tabular}{|c|c|c|c|c|}
\hline
Type & $\theta$                             & Conditions                                  & $\#$              & $\stab_{I(\sigma_1)}(\chi)$ \\ \hline

$1$  & $\mu_1 \otimes \mu_2 \otimes \rho$ & $\mu_i \in \Irr_0(R_i),$ $\mu_1 \neq \mu_2, $ & ${x-1 \choose 2}$ & $1$                       \\ \hline

$2$  & $\mu_1 \otimes \mu_2 \otimes \rho$ & $\mu_i \in \Irr_0(R_i),$ $\mu_1 = \mu_2, $    & $x-1$             & $\langle [c_\tau] \rangle$    \\ \hline

$3$  & $\nu_1 \otimes \nu_2 \otimes \rho$ & $\nu_i \in \Irr_1(R_i),$ $\nu_1 \neq \nu_2, $ & ${x \choose 2}$   & $1$                       \\ \hline

$4$  & $\nu_1 \otimes \nu_2 \otimes \rho$ & $\nu_i \in \Irr_1(R_i),$ $\nu_1 = \nu_2, $    & $x$               & $\langle [c_\tau] \rangle$    \\ \hline

$5$  & $\mu_1 \otimes \rho \otimes \mu_2$ & $\mu_i \in \Irr_0(R_i),$  & $(x-1)^2$ & $1$                       \\ \hline

$6$  & $\nu_1 \otimes \rho \otimes \nu_2$ & $\nu_i \in \Irr_1(R_i),$  & $x^2$ & $1$                       \\ \hline

\end{tabular}
\end{table}

\begin{table}[]
\renewcommand{\arraystretch}{1.4}
\centering
\caption{$I(\sigma_2)$-orbits of characters in case (b) and their stabilizers  }
\label{table:caseb}
\begin{tabular}{|c|c|c|c|c|}
\hline
Type & $\theta$                            & Conditions                                                                                                            & $\#$           & $\stab_{I(\sigma_2)}(\chi)$                             \\ \hline
$1$  & $\nu_1 \otimes \nu_2 \otimes \mu_3$ & $\nu_i \in \Irr_1(R_i),$ $\mu_3 \in \Irr_0(R_3), \nu_1 \neq \nu_2$ & $x(x-1)^2$     & $1$                                                   \\ \hline
$2$  & $\nu_1 \otimes \nu_2 \otimes \mu_3$ & $\nu_i \in \Irr_1(R_i),$ $\mu_3 \in \Irr_0(R_3), \nu_1 = \nu_2$                    & $2x(x-1)$      & $\langle [c_\tau] \rangle$                                \\ \hline
$3$  & $\nu_1 \otimes \mu_2 \otimes \nu_3$ & $\nu_i \in \Irr_1(R_i),$ $\mu_2 \in \Irr_0(R_2)$                                                                      & $2x^2(x-1)$    & $1$                                                   \\ \hline
$4$  & $\mu_1 \otimes \mu_2 \otimes \mu_3$ & $\mu_i \in \Irr_0(R_i)$, $\mu_1=\mu_2$                                                                                & $2(x-1)^2$     & $\langle [c_\tau] \rangle$                                \\ \hline
$5$  & $\mu_1 \otimes \mu_2 \otimes \mu_3$ & $\mu_i \in \Irr_0(R_i)$, $\mu_1 \neq \mu_2$                                                                           & $(x-1)^2(x-2)$ & $1$                                                   \\ \hline
\end{tabular}
\end{table}

\begin{table}[]
\renewcommand{\arraystretch}{1.4}
\centering
\caption{$I(\sigma_1)$-orbits of characters in case (b) and their stabilizers  }
\label{table:caseb2}
\begin{tabular}{|c|c|c|c|c|}
\hline
Type & $\theta$                            & Conditions                                                                                                            & $\#$           & $\stab_{I(\sigma_2)}(\chi)$                             \\ \hline
$1$  & $\nu_1 \otimes \nu_2 \otimes \mu_3$ & $\nu_i \in \Irr_1(R_i),$ $\mu_3 \in \Irr_0(R_3), \nu_1 \neq \nu_2$ & $x(x-1)^2$     & $1$                                                   \\ \hline

$2$  & $\nu_1 \otimes \nu_2 \otimes \mu_3$ & $\nu_i \in \Irr_1(R_i),$ $\mu_3 \in \Irr_0(R_3), \nu_1 = \nu_2$                    & $2x(x-1)$      & $\langle [c_\tau] \rangle$                                \\ \hline

$3$  & $\mu_1 \otimes \mu_2 \otimes \mu_3$ & $\mu_i \in \Irr_0(R_i)$, $\mu_1=\mu_2$                                                                                & $2(x-1)^2$     & $\langle [c_\tau] \rangle$                                \\ \hline

$4$  & $\mu_1 \otimes \mu_2 \otimes \mu_3$ & $\mu_i \in \Irr_0(R_i)$, $\mu_1 \neq \mu_2 \neq \mu_3 \neq \mu_1$                                                                           & ${x-1 \choose 3}/3$ & $1$                                                   \\ \hline

$5$  & $\mu_1 \otimes \mu_2 \otimes \mu_3$ & $\mu_i \in \Irr_0(R_i)$, $\mu_1 = \mu_2 = \mu_3$                                                                           & $2(x-1)$ & $I(\sigma_1)$                                                   \\ \hline

\end{tabular}
\end{table}

\end{proof}

Our general strategy is to reduce the computation of $\w_P(\D,0,d)$ to a few small values of $l$ by proving that  it is a polynomial in $x$. This is achieved in Corollary \ref{c:rat} below. We are most reliant on this approach in cases where we do not have an explicit description of the action of $\Out_\D(P)$ on $\Irr(P)$ such as when $\D=\F$, $P=C_S(E) \in \F^{cr}$ and $\Out_\D(P)=\GL_3(2)$. Here, the following lemma is relevant:

\begin{Lem}\label{l:gl32action}
For $k \ge 2$ let $V=(\mathbb{Z}/2^k)^3$ and $\varphi_k: G_{24} \rightarrow  \GL_3(V) \cong \GL_3(\mathbb{Z}/2^k)$ be the representations constructed in Section \ref{s:centrads}. For each subgroup $W \le \varphi_k(G_{24})$, let $\mathcal{O}$ be the set of orbits for the action of $W$ on $V$ (or $V^*$). For each $W_0 \le W$, the integer $|\{\alpha \in \mathcal{O} \mid \stab_W(\alpha)=W_0\}|$ can be expressed as a rational polynomial in $2^k$ of degree at most $3$.
\end{Lem}

\begin{proof}
Let $V, W$ and $\mathcal{O}$ be as in the lemma. Following \cite{orlik1982arrangements}, let $\mathcal{A}$ denote the set of $1$-eigenspaces of reflections in $V$ and denote by $L=L(\mathcal{A})$ the set of all intersections of elements of $\mathcal{A}$, regarded as a finite poset with minimal element $V$. If $X \in L$ let $L_X$ be the subposet $\{Y \in L \mid X \le Y\}$ with minimal element $X$, and let $$\chi(L_X,x):= \sum_{\substack{Y \in L \\ Y \ge X}} \mu(X,Y)x^{\dim(Y)}$$ be its Euler characteristic, where $\mu(X,Y)$ denotes the usual M\"{o}bius function for posets. Recall that $W' \le W$ is a parabolic subgroup if $W'=C_V(Y)$ for some $Y \in L$. By Steinberg's theorem (see \cite[Proposition 2.3]{KMS20}), we have $Y=C_V(W')$ and $$\bigcup_{W' \gneq W_0} C_V(W') = \bigcup_{\substack{W' \gneq W_0 \\ \mbox{parabolic}}} C_V(W').$$
We may thus assume that $W_0$ itself is a parabolic subgroup and define $X:=C_V(W_0)$. Let $n=n(W_0)$ be the the number of $W$-orbits with stabiliser $W_0$.  Counting, we have $$ n(W_0)=\frac{1}{|W:W_0|} |C_V(W_0) \backslash \bigcup_{W' > W_0} C_V(W')| = \frac{1}{|W:W_0|} |X \backslash \bigcup_{\substack{ Y > X \\ Y \in L}} Y| .$$  Writing $d(Y)=\dim(Y)$ for $Y \in L$, by inclusion-exclusion we have $$|W:W_0|n(W_0) = x^{d(X)}+x^{d(X)-1} \cdot \sum_{\substack{\sigma= (Y > X)\\ d(Y)=d(X)-1}} (-1)^{|\sigma|} + \cdots + x^{d(X)-s} \cdot \sum_{\substack{\sigma= (Y > \cdots > X)\\ d(Y)=d(X)-s}} (-1)^{|\sigma|} + \cdots  $$ where for $d(X) \ge s \ge 0$, each summation runs over chains $$\sigma =(Y=Y_0 > Y_1 > \cdots > Y_j = X) $$ in $L$ for which $d(Y)=d(X)-s$, and we regard the length $|\sigma|$ of such a chain to be $j$. Thus

$$|W:W_0|n(W_0) = \displaystyle\sum_{\sigma= (Y > \cdots > X)} x^{d(Y)}(-1)^{|\sigma|} = \displaystyle\sum_{\substack{Y \in L \\ Y \ge X}} \mu(X,Y)x^{d(Y)} = \chi(L_X,x),$$ a rational polynomial in $x$. Replacing $V$ with $V^*$ yields the analogous result for $V^*$ and the lemma is proved.  
\end{proof}

We remark that in the notation of Lemma \ref{l:gl32action} and its proof, if $W_0 \le W$ is a parabolic subgroup and $X=C_V(W_0)$, by the main result of \cite{orlik1982arrangements} we have $$\chi(L^X,x)=\prod_{k=1}^{\dim(X)} (x-b_k^X)$$ for integers $b_k^X$ which can be explicitly calculated using the information in \cite[Table 3]{orlik1982arrangements}.

\begin{Prop}\label{p:stabpoly}
Suppose that $l > 0$, $\D \in \{\H,\F\}$ and $P \in \D^{cr}$ has type $a$. For any $d \ge 0$ and $W \le \Out_\D(P)$ let $\mathcal{O}$ be the set of orbits for the action of $W$ on $\Irr^d(P)$. For each $W_0 \le W$, the integer $|\{\alpha \in \mathcal{O} \mid \stab_W(\alpha) = W_0\}|$  can be expressed as a rational polynomial in $x$ of degree at most $a$.
\end{Prop}

\begin{proof}
Let  $\D,P,a,W,W_0$ and $\mathcal{O}$ be as in the statement of the Proposition and set $\mathcal{O}_{W_0}:=\{\alpha \in \mathcal{O} \mid \stab_W(\alpha) = W_0\}$. We may assume that $a > 0$ since otherwise the result holds trivially.

Let $R_0 \le S$ be defined as in \cite[Notation 2.12]{lynd2017weights} so that $R_0$ is a (central) product of three generalized quaternion groups each of order $2^{l+3}$. Suppose first that $\Out_\mathcal{D}(P) \le \Out_\mathcal{K}(P)$ and let $\hat{W}$ be the preimage in $N_K(P)$ of $W$. By the discussion in Section \ref{s:centrads} we have, $P_0= P \cap R_0 = X_1X_2X_3$  is a product of three quaternion groups and we let $k$ be the number of these isomorphic with $Q_8$. Note that $P_0$ is a characteristic subgroup of $P$, $|P:P_0| \le 4$ and every irreducible character $\theta$ of $P_0$ extends to its inertia group $I_P(\theta)$. Thus by Lemma \ref{l:cliff}, there is a bijection 

$$\Phi: \bigcup \{(\theta, \beta) \mid \theta \in \Irr^{d_0}(P_0)/\hat{W} \mbox{ and } \beta \in \Irr^{d_1}(I_P(\theta)/P_0)\} \longrightarrow \mathcal{O},$$ where the union runs over all pairs  $(d_0,d_1)$ of integers such that $d_0+d_1=d$. Moreover $\Phi$ restricts to a bijection $$ \Phi_{W_0}: \bigcup\{(\theta, \beta) \mid I_W(\theta,\beta)=W_0, \theta \in \Irr^{d_0}(P_0)/\hat{W}, \beta \in \Irr^{d_1}(I_P(\theta)/P_0)\} \longrightarrow \mathcal{O}_{W_0}.$$ Now the action of $\Out_\mathcal{D}(P)$ on $\Irr(P_0)=\{\chi_1 \otimes \chi_2 \otimes \chi_3 \mid \chi_i \in \Irr(X_i)\}$  can be inferred from \cite[Theorem 3.12]{lynd2017weights}. From that description, for each $\beta \in \Irr^{d_1}(I_P(\theta)/P_0)$ the number of $\theta \in \Irr^{d_0}(P_0)/\hat{W}$ for which $I_W(\theta,\beta)=W_0$   can be explicitly determined using Lemma \ref{l:genquat} (c.f. Example \ref{ex:csu}). In particular, $|\O_{W_0}|$ can be expressed as a rational sum of binomial coefficients in $x$, a rational polynomial in $x$ of degree at most $3-k=a$.

Next suppose that $P \in \{C_S(E/Z), C_S(E)\}$ so that
$P$ contains the torus $T$ as explained in Section \ref{s:centrads}. Each $\theta \in \Irr(T)$ extends to its inertia group $I_P(\theta)$ so that by Lemma \ref{l:cliff}, there is a bijection $$\Phi: \{(\theta,\beta) \mid \theta \in \Irr(T)/P, \beta \in \Irr(I_P(\theta))/T\}.$$ By  Lemma \ref{l:gl32action}, for each  $\beta \in \Irr(I_P(\theta))/T$, the number of $W$-classes of elements $\theta \in \Irr(T)/P$ for which $I_W(\theta, \beta)=W_0$
 is a rational polynomial in $x$ of degree at most $3$, as needed. 
 
Finally we recall the explicit description of $P=R_{1^52}$ provided in Section \ref{s:centrads}. If $\chi=\chi_1 \otimes \chi_2 \otimes \chi_3 \in \Irr(Y_1 \times Y_2 \times Y_3)$ has $Z \le \ker(\chi)$ then either $\chi_i(1)=1$ for all $i$ or else $\chi_i(1)=2$ for one or three values of $i$, one of which is $3$. In this latter case either $\chi(1)=2$ and there are $x-1$ choices for $\chi_3$ or else $\chi(1)=8$ and there are $x$ choices. Since $W$ acts trivially on $\Irr(Y_3)$, we see that for all values of $d$, $|\mathcal{O}_{W_0}|$ is a polynomial in $x$ determined by the action of $\Out_\D(P)$ on the image of $Y_1 \times Y_2$ in $P$.

\end{proof}

\begin{Cor}\label{c:rat}
Suppose that $l > 0$, $\D \in \{\H,\F\}$ and $P \in \D^{cr}$ has type $a$. Then for each $d \ge 0$, $\w_P(\D,0,d)$ is a rational polynomial in $x$ of degree at most $a$.
\end{Cor}

\begin{proof}
Let $\D$ and $P$ be as in the statement and set $G:=\Out_\F(P)$. From Section \ref{s:main} we have

\begin{equation}\label{e:wqd}\w_P(\D,0,d):=
\sum_{\sigma \in \N_P/G} (-1)^{|\sigma|} 
\sum_{\mu \in \Irr^d(P)/I(\sigma)} z(k I(\sigma,\mu)).\end{equation} Fix $\sigma \in \N_P$ and let $\mathcal{O}$ be the set of orbits for the action of $W:=I(\sigma)$ on $\Irr^d(P)$. For each $W_0 \le W$, write $\mathcal{O}_{W_0}:=\{\alpha \in \mathcal{O} \mid \stab_W(\alpha)=W_0\}$. We have
  $$\sum_{\chi \in \Irr^d(P)/W} z(kI(\sigma, \chi)) = \sum_{W_0 \le W} z(kW_0) \cdot |\mathcal{O}_{W_0}|$$ is a rational polynomial in $x$ of degree at most $a$ by Proposition \ref{p:stabpoly}. Hence by (\ref{e:wqd}) the same is true of $\w_P(\D,0,d)$.
\end{proof}

We can now prove:

\begin{Thm}\label{t:table4}
For $l>0$, $\D \in \{\H,\F\}$ and $P \in \D^{cr}$ the polynomials $\w_P(\D,0,d)$ are listed in Table \ref{table:wpdod} according to the type of $P$.
\end{Thm}

\begin{proof}
Fix $d \ge 0$ and let $\D \in \{\H,\F\}$ and $P \in \D^{cr}$ have type $a$ for some $0 \le a \le 3$. By Corollary \ref{c:rat}, $\w_P(\D,0,d)$ is determined by the values it takes when $1 \le l \le a+1$. We calculate these values using MAGMA \cite{Magma}. To speed up our procedure we use the fact that, $\N_P$ can be replaced with the set $\E_P$ of elementary abelian chains in the definition of $\w_P(\D,0,d)$ (\cite[Lemma 4.13]{kessar2018weight}). The computation of $\w_P(\D,0,d)$ is broken down into the following stages:

\begin{itemize}
\item[(1)] Calculate the set $\E_P$ of non-empty elementary abelian chains of $p$-subgroups of $G$.
\item[(2)] For each $\sigma \in \E_P$, determine $I(\sigma) \le G$, and the set of conjugates $\sigma^G$; list a set of representatives $\E_P/G$ for $G$-conjugacy classes of chains.
\item[(3)] For each representative $\sigma \in  \E_P/G$ consider the action of $I(\sigma)$ on $\Irr^d(P)$.
\item[(4)] For each $I(\sigma)$-orbit $\mu \in \Irr^d(P)$, determine $C_{I(\sigma)}(\mu)$ and calculate $z(kC_{I(\sigma)}(\mu))$.
\end{itemize}

We briefly describe the MAGMA implementation. We consider the poset of all elementary abelian subgroups of $G$ to construct $\E_P$ as in (1). $I(\sigma)$ is calculated as a subgroup of the normalizer in $G$ of the largest element in $\sigma$, and then $\sigma^G$ is calculated using a transversal $[G/I(\sigma)]$. We use built-in commands to  construct the $I(\sigma)$-set $\Irr^d(P)$. Finally, $z(kC_{I(\sigma)}(\mu))$ is determined using the character table of $C_{I(\sigma)}(\mu)$ by applying \cite[Theorem 4.1]{lynd2017weights}.




\end{proof}

\begin{proof}[Proof of Theorem \ref{t:main}]
When $l > 0$, the polynomials in Table \ref{table:mdod} are given by summing the appropriate columns in Table \ref{table:wpdod}. It remains to check that for each $d \ge 0$ $\m(\D,0,d)$ evaluates to the correct number when $l=0$. For this we follow the strategy of the proof of Theorem \ref{t:table4} to calculate $\w_P(\D,0,d)$ for each $P \in \D^{cr}$. The results of this calculation are presented in Table \ref{table:wpdodl=0} and the result follows immediately by comparison with Table \ref{table:mdod}.
\end{proof}

\section{Verification of some conjectures}\label{s:conjectures}
In this section, we prove Theorems \ref{t:owcforspin} and \ref{t:conjectures}, beginning with the latter. We need a specific fact concerning $S$:

\begin{Prop}\label{p:conjnumber}
The following statements hold:
\begin{itemize}
\item[(1)] The number of conjugacy classes of $S$ is equal to $4x^3+15x^2+24x+18$.
\item[(2)] The number of conjugacy classes of $[S,S]$ is equal to $16x^3+12x$.
\end{itemize}
\end{Prop}

\begin{proof}
(1) follows by summing the entries in the first row of Table \ref{table:wpdod}. From the description of $S$ given in Section \ref{s:centrads}, we see that $[S,S]$ contains an abelian subgroup $T_0$ of index $2$ in terms of which elements of $\Irr([S,S])$ can be described using Lemma \ref{l:cliff}. Since $|T:T_0|=2$, $|\Irr([S,S])|$ is a polynomial in $x$ of degree at most $3$, and using this we calculate that it is equal to $16x^3+12x$.
\end{proof}

\begin{proof}[Proof of Theorem \ref{t:conjectures}]

Let $\D \in \{\H,\F\}$ and suppose that $\m(\D,0)=\k(\D,0)$. Inspecting the list of conjectures given in the theorem, we require:
\begin{itemize}
\item[(1)] $\k(\D,0) \le |S|$;
\item[(2)] $\m(\D,0,d) \ge 0$ for each positive integer $d$;
\item[(3)] $\m(\D,0,d) \neq 0$ for some $d \neq v_2(|S|)$;
\item[(4)] $\displaystyle\min_{r > 0} \{r: \Irr^{d-r}(S) \neq 0\} = \displaystyle\min_{r > 0} \{r: \m(\D,0,d-r) \neq 0\}$, where $d=v_2(|S|)$;
\item[(5)] $\k(\D,0)/\m(\D,0,d)$ is at most the number of conjugacy classes of $[S,S]$ for each $d \ge 0$;
\item[(6)] $\k(\D,0)/\w(\D,0)$ is at most the number of conjugacy classes of $S$.
\end{itemize}

All points follow from Table \ref{table:mdod}, where in (5) and (6) we also invoke Proposition \ref{p:conjnumber} and  the fact that $\w(\D,0)=12$ (\cite[Theorem 1.1]{lynd2017weights}). This completes the proof.
\end{proof}

We now turn our attention to Theorem \ref{t:owcforspin}. We briefly survey the background from Deligne--Lusztig theory we shall need. Let $G=\mathbf{G}^F$ be a finite reductive group of characteristic $r$. To a pair $(\mathbf{T},\theta)$ where $\mathbf{T}$ is an $F$-stable maximal torus of $\mathbf{G}$ and $\theta \in \Irr(\mathbf{T}^F)$ one can associate a generalized character $R_\mathbf{T}^\mathbf{G}(\theta) \in \Irr(G)$ (the exact construction of which we omit here). A notable property of these characters is that for every $\chi \in  \Irr(G)$ we have $\langle \chi, R_\mathbf{T}^\mathbf{G}(\theta) \rangle \neq 0$ for some pair $(\mathbf{T},\theta)$ as above. Letting $\mathbf{1} \in \Irr(\mathbf{T}^F)$ denote the trivial character, we say that $\chi \in  \Irr(G)$ is \textit{unipotent} if  $\langle \chi, R_\mathbf{T}^\mathbf{G}(\mathbf{1}) \rangle \neq 0$ for some $\mathbf{T}$ and we denote by $\Irr^u(G)$ the set of all unipotent characters. 

Now, for each $g \in G$ we have a Jordan decomposition $g=us=su$ of $g$ into its \textit{semisimple} and \textit{unipotent} parts $s$ and $u$ respectively. Let $\mathbf{G^*}$ denote a reductive group dual to $\mathbf{G}$, set $G^*=(\mathbf{G^*})^F$ and assume that $C_{G^*}(s)$ is connected for each semisimple element $s \in G^*$. By Jordan's decomposition of characters  there is a partition $\Irr(G) = \bigcup_s \E(G,s)$ where $s$ runs over a set of $G^*$-conjugacy classes of semisimple elements of $G^*$. Moreover each series $\E(G,s)$ is in bijection with the set of unipotent characters $\Irr^u(C_{G^*}(s))$. Under this bijection, for each $\lambda \in \Irr^u(C_{G^*}(s))$, we denote by $\chi_{s,\lambda}$ the corresponding element of $\Irr(G)$.  By \cite[Remark 13.24]{digne1991representations}, we have: \begin{equation}\label{e:chisl} \chi_{s,\lambda}(1)=|G^*:C_{G^*}(s)|_{r'}. \lambda(1),\end{equation}
Here is the main result we shall need:
\begin{Prop}\label{p:cabanes}
Suppose that $G=\Spin_7(q)=\mathbf{B}_3(q)$ and let $B$ be the principal $2$-block of $G$. Then $\Irr(B)= \bigcup_s \E(G,s),$
where $s$ runs over a complete set of $G$-conjugacy class representatives of $2$-elements of the dual group $G^*=\Aut(\PSp_6(q))=\mathbf{C}_3(q)$.
\end{Prop}

\begin{proof}
See \cite[Theorem 21.14]{cabanes2004representation}.
\end{proof}

We can now prove Theorem \ref{t:owcforspin}:

\begin{proof}[Proof of Theorem \ref{t:owcforspin}]
We use the explicit enumeration of the characters in the principal $2$-block of $\Spin_7(q)$ provided by Proposition \ref{p:cabanes}. The $2$-elements in $\mathbf{C}_3(q)$, and the degrees of unipotent characters of their centralizers are calculated using CHEVIE \cite{geck1996chevie}. The relevant information is summarised in Tables \ref{table:1mod4} and \ref{table:3mod4} for $q \equiv 1 \mod 4$ and $q \equiv 3 \mod 4$ respectively. Precisely, the $2$-elements $s$ are enumerated according to the types of their centralizers, with the number of classes of a given type being given in the third column. For each $s$ we list $|G^*:C_{G^*}(s)|$ and the degrees of the unipotent characters of $C_{G^*}(s)$ in the fourth and fifth columns in terms of cyclotomic polynomials $\phi_n:=\phi_n(q)$. This information is combined to list, in the final column, the numbers $v_2(\lambda(1)|G^*:C_{G^*}(s)|)$ for each $\lambda \in \Irr^u(C_{G^*}(s))$ in terms of $l.$ In particular, $$v_2(\phi_1)=\begin{cases} l+2 & \mbox{ if $q \equiv 1 \mod 4$ } \\ 1 & \mbox{ if $q \equiv 3 \mod 4$ } \end{cases} \mbox{ and } v_2(\phi_2)=\begin{cases} 1 & \mbox{ if $q \equiv 1 \mod 4$ } \\ l+2 & \mbox{ if $q \equiv 3 \mod 4$, } \end{cases}$$
and for all $q$ and all other $\phi_n$ of interest here,  $v_2(\phi_n)=0$ unless $n=4$ where $v_2(\phi_4)=1$.
It is now easy to check from the tables that the number of characters of defect $d$ coincides with the integer $\m(\H,0,d)$ listed in Table \ref{table:mdod} and Theorem \ref{t:owcforspin} follows.
\end{proof}



\section{Local-global correspondences for exotic fusion systems}\label{s:conclude}



\subsection{$p$-compact groups}\label{ss:pcomp}

A connected $p$-compact group is a triple $(\mathbf{X},B\mathbf{X},e)$ where $\mathbf{X}$ is a connected space with $H^*(\mathbf{X};\mathbb{F}_p)$ finite, $B\mathbf{X}$ is a pointed $p$-complete space, and $e : \mathbf{X} \rightarrow \Omega B\mathbf{X}$ is a weak homotopy equivalence. It is a deep theorem in homotopy theory that connected $p$-compact groups are in 1-1 correspondence with $\mathbb{Z}_p$-root data (see \cite[Theorem 2.3]{grodal2010classification}.) When $p$ is odd a root datum is completely determined by the corresponding finite $\mathbb{Z}_p$-reflection group  (see \cite{andersen2008classification}); when $p=2$, there is a unique \textit{exotic} $2$-compact group $\DI(4)$ first constructed by Dwyer--Wilkerson \cite{DwyerWilkerson1993} determined by the $\mathbb{Z}_2$-reflection group $G_{24} \cong \GL_3(2) \times 2$ (see \cite{andersen2009classification}). If $p$ is odd, $\mathbf{X}$ is a simply connected $p$-compact group and $q$ is  prime power (prime to $p$), then by \cite[Theorem A]{BrotoMoller2007} the space $\mathbf{X}(q)$ of homotopy fixed points under the unstable Adams operation $\psi^q$ on $\mathbf{X}$ is the classifying space of a saturated fusion system $\F(q)$. These are the fusion systems of the so-called \textit{Chevalley $p$-local finite groups}. When $p=2$, the fusion systems $\Sol(q)$ are also obtained in this way, by taking $\mathbf{X}$ to be the $2$-compact group $\DI(4)$ (see \cite{LeviOliver2002}, \cite{LeviOliver2005}.) The above discussion motivates the following generalization of Theorem \ref{t:main} to all primes $p$:

\begin{Conj}\label{c:main}
Let $\mathbf{X}$ be a (simply) connected $p$-compact group, $q$ be a prime power prime to $p$ and $\F(q)$ be the saturated $p$-fusion system associated to the space of homotopy fixed points $\mathbf{X}(q)$ of the unstable Adams operation $\psi^q$. Let $\alpha$ be a compatible K\"{u}lshammer--Puig family associated to $\F(q).$  Then for each $d \ge 0$, $\m(\F(q),\alpha,d)$ is a rational polynomial in $p^l$ where $l=v_p(1-q^k)$ and $$k:=\begin{cases} \ord_p(q) & \mbox{ if $p > 2$; } \\ 2 & \mbox{ if $p=2$,} \end{cases}$$ where $\ord_p(q)$ the order of $q$ mod $p$.
\end{Conj}

When $(\F(q),\alpha)$ is realised by a block $B$ of a finite  group of Lie type of characteristic $q$, one expects Conjecture \ref{c:main} to follow from Robinson's Conjecture \ref{c:rob} and a result in Deligne--Lusztig theory along the lines of Proposition \ref{p:cabanes}. When $(\F(q),\alpha)$ is exotic, the possibilities for $\F(q)$ are listed in \cite{BrotoMoller2007}.  In these cases, Conjecture \ref{c:main} would follow from \cite[Conjecture 6.13]{KMS20}.

\subsection{Spetses}\label{ss:spetses}

Using the results of the present paper we give some evidence in support of Conjecture \ref{c:main} using the Brou\'{e}--Malle--Michel theory of spetses \cite{broue1999towards}. That is, we treat the topological space $\mathbf{X}:=\DI(4)$ as if it were a connected reductive algebraic group over $\bar{\mathbb{F}}_q$ with Weyl group isomorphic to $G_{24}$, and consider its $2$-local structure. A spets has a naturally defined set of unipotent characters (see \cite{broue1997complex}). We set $\Irr^u(\mathbf{X}(q))$ equal to this set of twenty-two characters, the degrees of which are listed in \cite[Section A.9]{broue2014split}. Note, that in general these degrees do not have coefficients in $\mathbb{Q}$ but rather over the finite extension  
$\mathbb{Q}[\sqrt{-7}]$ of $\mathbb{Q}$. In this case, writing $\sigma:=(1-\sqrt{-7})/2$,  define $$\phi_7'=q^3+\sigma q^2-\bar{\sigma}q-1, \hspace{2mm} \phi_{14}'=q^3-\sigma q^2-\bar{\sigma}q+1, \hspace{2mm} \phi_7''=\bar{\phi_7'} \hspace{2mm} \mbox{ and } \hspace{2mm} \phi_{14}''=\bar{\phi_{14}'},$$ where $\bar{\cdot}$ is Galois conjugation in $\mathbb{Q}[\sqrt{-7}]$, and note that $\phi_7$ and $\phi_{14}$ decompose as $\phi_7'\phi_7''$ and $\phi_{14}'\phi_{14}''$ respectively.

 For the remaining `spetsial characters' of $\mathbf{X}(q)$ we take inspiration from Deligne--Lusztig theory or, more precisely, Proposition \ref{p:cabanes}. Our starting point is the fact that $C_\F(s) \subseteq \H$ for each non-trivial fully $\F$-centralized element $s \in S$. Hence by \cite[Lemma I.1.2]{AschbacherKessarOliver2011} $C_\F(s)=C_\H(s)$ is a saturated fusion system realized by $C_H(s)$. 
Since $G_{24}$ has connection index $1$ (see \cite[App. B]{broue2018cyclotomic}) we treat $\mathbf{X}$ as if it were Langlands self-dual and define a `Lusztig series' $\E(\mathbf{X}(q),s)$ and a set $\Irr(\mathbf{X}(q))$ of irreducible characters as in (\ref{e:fls}):

\begin{equation} \E(\mathbf{X}(q),s):=\begin{cases} \Irr^u(\mathbf{X}(q)) & \mbox{ if $s=1$;} \\ \E(H,s) & \mbox{ otherwise, }
 \end{cases} \mbox{ and } \Irr(\mathbf{X}(q))=\bigcup_s \E(\mathbf{X}(q),s).\end{equation}
 
This motivates Theorem \ref{t:exowc}, which we now prove. Note that since the ring $\mathbb{Q}[\sqrt{-7}]$ is a UFD, the $2$-valuations of $\phi_7',\phi_7'',\phi_{14}'$ and $\phi_{14}''$ make sense (and in any case are all zero).


\begin{proof}[Proof of Theorem \ref{t:exowc}] We require a complete set of $\F$-conjugacy class representatives of elements of $S$. We first obtain such a set for $\H$ using CHEVIE. We then determine which classes are fused in $\F$. If $x,y \in T$ are fully $\F$-centralized then $x$ and $y$ are $\F$-conjugate if and only if they are $\Aut_\F(T)$-conjugate. If $G_x$ realises $C_\F(x)$ then the Weyl group of $G_x$ is isomorphic to $\Stab_{\Aut_\F(T)}(x)$. We use MAGMA to enumerate orbits on $T$ under the action of $\Aut_\F(T)$ for $0 \le l \le 3$ from which we determine polynomial expressions for the number of class representatives in $T$ with a given type of centralizer. The classes of elements outside of $T$ are distinguished according to the type of their centralizer which is $\mathbf{A}_1(q^2)$, $\mathbf{A}_1(q)$ or $\A_0(q)$. For notational convenience write $|\mathbf{X}(q): C_{\mathbf{X}(q)}(s)|$ for the index $|H:C_H(s)|$ and set $\Irr^u(C_{\mathbf{X}(q)}(s)):=\Irr^u(C_H(s))$ whenever $s \neq 1$. Tables \ref{table:spets} and \ref{table:spets2} list the unipotent degrees of elements of  $\bigcup_s \E(\mathbf{X}(q),s)$ together with their $2$-adic valuations. By comparison with Table \ref{table:mdod} we observe that for each $d > 0$, $\m(\F,0,d)$ is exactly the number of characters in $\bigcup_s \E(\mathbf{X}(q),s)$ of defect $d$, as needed. 
\end{proof}

\newpage




\section*{Tables}




%

\begin{table}[H]
\renewcommand{\arraystretch}{1.5}\tiny
\centering
\caption{$\w_P(\D,0,d)$ for $\D \in \{\H,\F\}$ and all $d \ge 0$ when $l > 0$}
\label{table:wpdod}
\begin{tabular}{|c|c|c|c|c|c|c|c|c|c|}
\hline
   $P$ & $\D$  & $a$  & $al+4$  & $al+5$  & $al+6$ & $al+7$      & $al+8$       & $al+9$   & $al+10$  \\ \hline \hline
$S$   & $\H,\F$ &  $3$        & $-$        & $-$         & $x(4x^2-5x+2)$ & $2(10x^2-8x+1)$ & $6(5x-2)$   & $4(2x+3)$   & $16$       \\ \hline

$C_S(U)$                 & $\F$ & $3$               & $-$        & $-$             & $-\frac{1}{3}(8x^3-6x^2-8x+9)$    &$-2(4x^2-x-1)$    & $-8x$ & $0$ & $-$        \\ \hline
$C_S(E/Z)$                   & $\H,\F$        & $3$                     & $-$ & $-$ & $-\frac{8}{3} x^3+2x^2+\frac{5}{3}x-1$            & $-8x^2+8x$             & $-16x+8$         & $-8$         & $-$        \\ \hline
$C_S(E)$                 & $\F$        & $3$        & $-$        & $-$                    &               $\frac{32}{21}x^3-\frac{8}{3}x+\frac{8}{7}$               & $0$            &   $-$      & $-$     & $-$        \\ \hline

$R_1R_2Q_3\langle \tau \rangle $ & $\H,\F$  & $2$ & $-$ & $x(x-1)$   & $4x$       & $0$        & $0$        & $0$                   & $-$        \\ \hline

$Q_1R_2R_3$  & $\H$ & $2$ & $-$      & $2x(x-1)$ & $8x$ & $0$ & $0$ & $-$ & $-$                   \\ \hline

$Q_1Q_2R_3$   &  $\H,\F$ & $1$ & $-$       & $0$        & $0$        & $2(x-1)$        & $8$            & $-$ & $-$           \\ \hline

$Q_1R_2Q_3$ & $\H$ & $1$                 &$-$ & $x-1$ & $4$ & $0$ & $0$ & $-$    & $-$            \\ \hline

$Q_1Q_2'R_3$   &  $\H,\F$ & $1$ & $-$       & $0$        & $0$        & $2(x-1)$        & $8$            & $-$ & $-$           \\ \hline

$Q_1R_2Q_3'$ & $\H$ & $1$                 &$-$ & $x-1$ & $4$ & $0$ & $0$ & $-$    & $-$            \\ \hline

$Q_1Q_2R_3\langle \tau \rangle $   & $\H,\F$ & $1$         & $-$                     & $-$              & $2(x-1)$   & $10-2x$    & $-8$           & $0$  & $-$           \\ \hline

$Q_1Q_2'R_3\langle \tau' \rangle $ & $\H,\F$ & $1$       & $-$                     & $-$                & $2x-1$   & $6-2x$    & $-8$           & $0$  & $-$          \\ \hline

$R_{1^52}$ & $\H,\F$ & $1$                      &$0$ &             $-$           & $-2(x-1)$           & $-8$   &$-$         & $-$     &  $-$       \\ \hline

$Q_1Q_2Q_3$   &    $\H,\F$                 & $0$         & $-$ & $-$ & $0$ & $-$ & $0$ & $-$       & $-$         \\ \hline

$Q_1'Q_2Q_3$ &      $\H,\F$                 & $0$         & $-$ & $-$ & $0$ & $-$ & $0$ & $-$       & $-$         \\ \hline

$Q_1Q_2Q_3'$   &    $\H$                 & $0$         & $-$ & $-$ & $0$ & $-$ & $0$ & $-$       & $-$         \\ \hline 

$Q_1Q_2Q_3\langle \tau \rangle $   & $\H,\F$ & $0$ & $-$         & $-$                     & $0$                     & $0$                     & $0$                     & $0$       &$-$           \\ \hline

$Q_1Q_2Q_3'\langle \tau \rangle $   & $\H,\F$ & $0$ & $-$         & $-$                     & $0$                     & $0$                     & $0$                     & $0$       &$-$           \\ \hline

$R_{1^7}$   & $\H,\F$ & $0$ &     $0$       & $-$  &     $-$     &     $0$                    &  $-$                    & $-$  &  $-$         \\ \hline

$R'_{1^7}$    & $\H,\F$ & $0$                  & $2$            &  $-$                        &      $0$                   &     $-$                    &                        $-$ &  $-$  &  $-$               \\ \hline   

$A$             & $\F$ & $0$           & $0$ & $-$                     & $-$                     & $-$                     & $-$                     & $-$      & $-$       \\ \hline

\end{tabular}
\end{table}

\begin{table}[]
\renewcommand{\arraystretch}{1.4}
\centering
\caption{$\w_P(\D,0,d)$ for $\D \in \{\H,\F\}$ and all $d \ge 0$ when $l = 0$}
\label{table:wpdodl=0}
\begin{tabular}{|c|c|c|c|c|c|c|c|c|}
\hline
   $P$ & $\D$   & $4$  & $5$  & $6$ & $7$      & $8$       & $9$   & $10$  \\ \hline
\hline
$S$  & $\H,\F$            & $-$         & $-$ & $1$         & $6$          & $18$         & $20$ & $16$ \\ \hline
$Q\langle \tau' \rangle$   & $\H,\F$        & $-$         & $-$ & $1$         & $4$          & $-8$         & $0$  & $-$  \\ \hline
$C_S(E/Z)$   & $\H,\F$        & $-$         & $-$ & $0$         & $0$          & $-8$         & $-8$ & $-$  \\ \hline
$Q\langle \tau \rangle$     & $\H,\F$        & $-$         & $-$ & $4$         & $8$          & $-8$         & $0$  & $-$  \\ \hline
$Q$      & $\H,\F$        & $-$         & $-$ & $16$, $0$ & $-$          & $16$ & $-$  & $-$  \\ \hline
$R_{1^7}$      & $\H,\F$        & $0$		 & $-$ & $-$         & $-8$ & $-$          & $-$  & $-$  \\ \hline
$R'_{1^7}$    & $\H,\F$        & $2$ & $-$ & $0$ & $-$          & $-$          & $-$  & $-$  \\ \hline
$A$      &  $\F$     & $0$ & $-$ & $-$ & $-$          & $-$          & $-$  & $-$  \\ \hline
$C_S(E)$   &  $\F$         & $-$ & $-$ & $0$ & $0$          & $-$          & $-$  & $-$  \\ \hline
$C_S(U)$    &  $\F$        & $-$ & $-$ & $-1$ & $-4$          & $-8$          & $0$  & $-$  \\ \hline
\hline
$\m(\D,0,d)$ & $\H,\F$ & $2$         & $0$ & $22$, $5$     & $10$, $6$ &$10$, $2$         & $12$ & $16$ \\ \hline
\end{tabular}
\end{table}

\begin{landscape}
\begin{table}[]
\renewcommand{\arraystretch}{1.4}\tiny
\centering
\caption{Degrees of irreducible characters in the principal $2$-block of $\Spin_7(q)$, $q \equiv 1 \mod 4$ }
\label{table:1mod4}
\begin{tabular}{|c|c|c|c|c|}
\hline
 type of $C_{G^*}(s)$ & $\#$ classes                          & $|G^*:C_{G^*}(s)|$                                     & $\lambda(1), \lambda \in \Irr^u(C_{G^*}(s))$                                                                                                                                                                                                                               & $v_2(\lambda(1)|G^*:C_{G^*}(s)|), \lambda \in \Irr^u(C_{G^*}(s)) $ \\ \hline
 $\C_3(q)$              & $1$                                 & $1$                                              & $\makecell{\frac{1}{2}q^4\phi_4 \phi_6, q^3 \phi_3 \phi_6, \frac{1}{2}q^4 \phi_3 \phi_4, q^9, \frac{1}{2}q \phi_2^2 \phi_6, q^2 \phi_3 \phi_6, \\ \frac{1}{2} q \phi_3 \phi_4, \frac{1}{2}q^4\phi_2^2 \phi_6,1, \frac{1}{2}q\phi_4\phi_6,\frac{1}{2}q \phi_1^2\phi_3,\frac{1}{2}q^4\phi_1^2\phi_3}$ &    $0,0,0,0,1,0,0,1,0,0,2l+3,2l+3$     \\ \hline
 $\C_2(q) + \A_1(q)$     & $1$                                 & $q^4 \phi_3 \phi_6$                              & $\makecell{1,\frac{1}{2}q\phi_1^2,\frac{1}{2}q\phi_2^2,\frac{1}{2}q\phi_4,\frac{1}{2}q\phi_4,q^4, \\ q,\frac{1}{2}q^2\phi_1^2, \frac{1}{2}q^2\phi_2^2,\frac{1}{2}q^2\phi_4,\frac{1}{2}q^2\phi_4,q^5}$                                                      &    $0,2l+3,1,0,0,0,0,2l+3,1,0,0,0$     \\ \hline
 $\C_2(q)$              & $2x - 1$                         & $q^5 \phi_2 \phi_3 \phi_6$                       & $1,\frac{1}{2}q\phi_1^2, \frac{1}{2}q\phi_2^2,\frac{1}{2}q\phi_4,\frac{1}{2}q\phi_4,q^4$                                                                                                                            &   $1,2l+4,2,1,1,1$      \\ \hline
 $\A_1(q) + ~\A_1(q)$    & $2x - 1$                         & $q^7 \phi_2 \phi_3 \phi_4 \phi_6$                & $1,q,q,q^2$                                                                                                                                                                                                                           &    $2,2,2,2$     \\ \hline
 $\A_1(q) + \A_1(q)$     & $x - 1$                         & $q^7 \phi_2 \phi_3 \phi_4 \phi_6$                & $1,q,q,q^2$                                                                                                                                                                                                                           &   $2,2,2,2$      \\ \hline
 $\A_1(q^2)$            & $x$                             & $q^7 \phi_1 \phi_2^2 \phi_3 \phi_6$              & $1,q^2$                                                                                                                                                                                                                               &   $l+4,l+4$      \\ \hline
 $(\A_1(q) + \A_1(q)).2$ & $1$                                 & $\frac{1}{2} q^7 \phi_2 \phi_3 \phi_4 \phi_6$          & $1,1,q^2,q^2,2q$                                                                                                                                                                                                                      &  $1,1,1,1,2$       \\ \hline
 $\A_1(q^2).2$          & $1$                                 & $\frac{1}{2}q^7 \phi_1^2 \phi_2 \phi_3 \phi_6$         & $1,1,q^2,q^2$                                                                                                                                                                                                                         &  $2l+4,2l+4,2l+4,2l+4$       \\ \hline
 $\A_2(q)$             & $2x - 1$                         & $q^6 \phi_2^2 \phi_4 \phi_6$                     & $1,q\phi_2,q^3$                                                                                                                                                                                                                       &    $3,4,3$     \\ \hline
 $~\A_2(q).2$           & $1$                                 & $\frac{1}{2} q^6 \phi_2^2 \phi_4 \phi_6$               & $1,1,q\phi_2,q\phi_2,q^3,q^3$                                                                                                                                                                                                         &     $2,2,3,3,2,2$    \\ \hline
 $^2\A_2(q).2$         & $1$                                 & $\frac{1}{2} q^6 \phi_1^2 \phi_3 \phi_4$                & $1,1,q\phi_1,q\phi_1,q^3,q^3$                                                                                                                                                                                                         &  $2l+4,2l+4,3l+6,3l+6,2l+4,2l+4$       \\ \hline
 $\A_1(q)$              & $2x^2-3x + 1$            & $q^8 \phi_2^2 \phi_3 \phi_4 \phi_6$              & $1,q$                                                                                                                                                                                                                                 &   $3,3$      \\ \hline
 $\A_1(q)$              & $x$                             & $q^8 \phi_1 \phi_2 \phi_3 \phi_4 \phi_6$         & $1,q$                                                                                                                                                                                                                                 &    $l+4,l+4$     \\ \hline
 $~\A_1(q)$             & $4x^2-4x + 1$                   & $q^8 \phi_2^2 \phi_3 \phi_4 \phi_6$              & $1,q$                                                                                                                                                                                                                                 &   $3,3$      \\ \hline
 $~\A_1(q)$             & $x$                             & $q^8 \phi_1 \phi_2 \phi_3 \phi_4 \phi_6$         & $1,q$                                                                                                                                                                                                                                 &   $l+4,l+4$      \\ \hline
 $~\A_1(q)$             & $x$                             & $q^8 \phi_1 \phi_2 \phi_3 \phi_4 \phi_6$         & $1,q$                                                                                                                                                                                                                                 &   $l+4,l+4$      \\ \hline
$\A_0(q)$              & $\frac{4}{3} x^3 - 3x^2 + \frac{5}{3} x$ & $q^9 \phi_2^3 \phi_3 \phi_4 \phi_6$              & $1$                                                                                                                                                                                                                                   &     $4$    \\ \hline
 $\A_0(q)$              & $x^2-x$               & $q^9 \phi_1 \phi_2^2 \phi_3 \phi_4 \phi_6$       & $1$                                                                                                                                                                                                                                   &     $l+5$    \\ \hline
 $\A_0(q)$              & $2x^2-2x$                & $q^9 \phi_1 \phi_2^2 \phi_3 \phi_4 \phi_6$       & $1$                                                                                                                                                                                                                                   &    $l+5$     \\ \hline
$\A_0(q).2$            & $x-1$                         & $\frac{1}{2} q^9 \phi_2^3 \phi_3 \phi_4 \phi_6$        & $1,1$                                                                                                                                                                                                                                 &   $3,3$      \\ \hline
$\A_0(q).2$            & $x$                             & $\frac{1}{2}q^9 \phi_1 \phi_2^2 \phi_3 \phi_4 \phi_6$  & $1,1$                                                                                                                                                                                                                                 &  $l+4,l+4$       \\ \hline
 $\A_0(q).2$            & $x$                             & $\frac{1}{2} q^9 \phi_1 \phi_2^2 \phi_3 \phi_4 \phi_6$ & $1,1$                                                                                                                                                                                                                                 &  $l+4,l+4$       \\ \hline
 $\A_0(q).2$            & $x-1$                         & $\frac{1}{2} q^9 \phi_1^2 \phi_2 \phi_3 \phi_4 \phi_6$ & $1,1$                                                                                                                                                                                                                                 &   $2l+5,2l+5$      \\ \hline
\end{tabular}
\end{table}

\end{landscape}

\begin{landscape}
\begin{table}[]
\renewcommand{\arraystretch}{1.4}\tiny
\centering
\caption{Degrees of irreducible characters in the principal $2$-block of $\Spin_7(q)$, $q \equiv 3 \mod 4$ }
\label{table:3mod4}
\begin{tabular}{|c|c|c|c|c|}
\hline
type of $C_{G^*}(s)$ & $\#$ classes                          & $|G^*:C_{G^*}(s)|$                                     & $\lambda(1), \lambda \in \Irr^u(C_{G^*}(s))$                                                                                                                                                                                                                               & $v_2(\lambda(1)|G^*:C_{G^*}(s)|), \lambda \in \Irr^u(C_{G^*}(s)) $ \\ \hline

 $\C_3(q)$              & $1$                                 & $1$                                              & $\makecell{\frac{1}{2}q^4\phi_4 \phi_6, q^3 \phi_3 \phi_6, \frac{1}{2}q^4 \phi_3 \phi_4, q^9, \frac{1}{2}q \phi_2^2 \phi_6, q^2 \phi_3 \phi_6, \\ \frac{1}{2} q \phi_3 \phi_4, \frac{1}{2}q^4\phi_2^2 \phi_6,1, \frac{1}{2}q\phi_4\phi_6,\frac{1}{2}q \phi_1^2\phi_3,\frac{1}{2}q^4\phi_1^2\phi_3}$ &    $0,0,0,0,2l+3,0,0,2l+3,0,0,1,1$     \\ \hline

 $\C_2(q) + \A_1(q)$     & $1$                                 & $q^4 \phi_3 \phi_6$                              & $\makecell{1,\frac{1}{2}q\phi_1^2,\frac{1}{2}q\phi_2^2,\frac{1}{2}q\phi_4,\frac{1}{2}q\phi_4,q^4, \\ q,\frac{1}{2}q^2\phi_1^2, \frac{1}{2}q^2\phi_2^2,\frac{1}{2}q^2\phi_4,\frac{1}{2}q^2\phi_4,q^5}$                                                      &    $0,1,2l+3,0,0,0,0,1,2l+3,0,0,0$     \\ \hline

 $\C_2(q)$              & $2x-1$                         & $q^5 \phi_1 \phi_3 \phi_6$                       & $1,\frac{1}{2}q\phi_1^2, \frac{1}{2}q\phi_2^2,\frac{1}{2}q\phi_4,\frac{1}{2}q\phi_4,q^4$                                                                                                                            &   $1,2,2l+4,1,1,1$       \\ \hline

 $\A_1(q) + ~\A_1(q)$    & $2x-1$                         & $q^7 \phi_1 \phi_3 \phi_4 \phi_6$                & $1,q,q,q^2$                                                                                                                                                                                                                           &     $2,2,2,2$     \\ \hline

 $\A_1(q) + \A_1(q)$     & $x - 1$                         & $q^7 \phi_1 \phi_3 \phi_4 \phi_6$                & $1,q,q,q^2$                                                                                                                                                                                                                           &   $2,2,2,2$      \\ \hline

 $\A_1(q^2)$            & $x$                             & $q^7 \phi_1^2 \phi_2 \phi_3 \phi_6$              & $1,q^2$                                                                                                                                                                                                                               &  $l+4,l+4$        \\ \hline

$(\A_1(q) + \A_1(q)).2$ & $1$                                 & $\frac{1}{2} q^7 \phi_1 \phi_3 \phi_4 \phi_6$          & $1,1,q^2,q^2,2q$                                                                                                                                                                                                                      &  $1,1,1,1,2$       \\ \hline

 $\A_1(q^2).2$          & $1$                                 & $\frac{1}{2}q^7 \phi_1 \phi_2^2 \phi_3 \phi_6$         & $1,1,q^2,q^2$                                                                                                                                                                                                                         &     $2l+4,2l+4,2l+4,2l+4$     \\ \hline

 $^2\A_2(q)$             & $2x - 1$                         & $q^6 \phi_1^2\phi_3 \phi_4$                     & $1,q\phi_1,q^3$                                                                                                                                                                                                                       &    $3,4,3$     \\ \hline

 $^2\A_2(q).2$         & $1$                                 & $\frac{1}{2} q^6 \phi_1^2 \phi_3 \phi_4$                & $1,1,q\phi_1,q\phi_1,q^3,q^3$                                                                                                                                                                                                         &   $2,2,3,3,2,2$        \\ \hline

 $~\A_2(q).2$           & $1$                                 & $\frac{1}{2} q^6 \phi_2^2 \phi_4 \phi_6$               & $1,1,q\phi_2,q\phi_2,q^3,q^3$                                                                                                                                                                                                         &     $2l+4,2l+4,3l+6,3l+6,2l+4,2l+4$    \\ \hline

 $\A_1(q)$              & $2x^2-3x+1$                             & $q^8 \phi_1^2 \phi_3 \phi_4 \phi_6$         & $1,q$                                                                                                                                                                                                                                 &  $3,3$       \\ \hline

 $\A_1(q)$              & $x$            & $q^8 \phi_1 \phi_2 \phi_3 \phi_4 \phi_6$              & $1,q$                                                                                                                                                                                                                                 & $l+4,l+4$      \\ \hline

 $~\A_1(q)$             & $4x^2 - 4x + 1$                   & $q^8 \phi_1^2 \phi_3 \phi_4 \phi_6$              & $1,q$                                                                                                                                                                                                                                 &   $3,3$     \\ \hline

 $~\A_1(q)$             & $x$                             & $q^8 \phi_1 \phi_2 \phi_3 \phi_4 \phi_6$         & $1,q$                                                                                                                                                                                                                                 &   $l+4,l+4$     \\ \hline

 $~\A_1(q)$             & $x$                             & $q^8 \phi_1 \phi_2 \phi_3 \phi_4 \phi_6$         & $1,q$                                                                                                                                                                                                                                 &  $l+4,l+4$       \\ \hline

 $\A_0(q)$              & $\frac{4}{3} x^3 - 3 x^2 + \frac{5}{3} x$ & $q^9 \phi_1^3 \phi_3 \phi_4 \phi_6$              & $1$                                                                                                                                                                                                                                   &     $4$    \\ \hline

$\A_0(q)$              & $x^2-x$               & $q^9 (\phi_1)^2 \phi_2 \phi_3 \phi_4 \phi_6$       & $1$                                                                                                                                                                                                                                   &   $l+5$      \\ \hline

 $\A_0(q)$              & $2x^2-2x$                & $q^9 \phi_1^2 \phi_2 \phi_3 \phi_4 \phi_6$       & $1$                                                                                                                                                                                                                                   & $l+5$        \\ \hline

 $\A_0(q).2$            & $x - 1$                         & $\frac{1}{2} q^9 \phi_1^3 \phi_3 \phi_4 \phi_6$ & $1,1$                                                                                                                                                                                                                                 &   $3,3$       \\ \hline

 $\A_0(q).2$            & $x$                             & $\frac{1}{2}q^9 \phi_1^2 \phi_2 \phi_3 \phi_4 \phi_6$  & $1,1$                                                                                                                                                                                                                                 &   $l+4,l+4$    \\ \hline

 $\A_0(q).2$            & $x$                             & $\frac{1}{2} q^9 \phi_1^2 \phi_2 \phi_3 \phi_4 \phi_6$ & $1,1$                                                                                                                                                                                                                                 &    $l+4,l+4$      \\ \hline

 $\A_0(q).2$            & $x - 1$                         & $\frac{1}{2} q^9 \phi_1 \phi_2^2 \phi_3 \phi_4 \phi_6$        & $1,1$                                                                                                                                                                                                                                 &  $2l+5,2l+5$    \\ \hline

\end{tabular}
\end{table}

\end{landscape}

\begin{landscape}
\begin{table}[]
\renewcommand{\arraystretch}{1.4}\tiny
\centering
\caption{Degrees of irreducible characters in $\Irr(\mathbf{X}(q))$, $q \equiv 1 \mod 4$ }
\label{table:spets}
\begin{tabular}{|c|c|c|c|c|}
\hline
 type of $C_{\mathbf{X}(q)}(s)$ & $\#$ classes         & $|\mathbf{X}(q):C_{\mathbf{X}(q)}(s)|$                                     & $\lambda(1), \lambda \in \Irr^u(C_{\mathbf{X}(q)}(s))$                                                                                                                                                                                                                               & $v_2(\lambda(1)|\mathbf{X}(q):C_{\mathbf{X}(q)}(s)|)$ \\ \hline

\hline
$\mathbf{X}(q)$            & $1$         & $1$ & 

$\makecell{1,
\frac{\sqrt{-7}}{14}q \phi_3 \phi_4 \phi_6 \phi_7' \phi_{14}'', 
\frac{-\sqrt{-7}}{14}q \phi_3 \phi_4 \phi_6 \phi_7'' \phi_{14}', \\ 
\frac{1}{2}q \phi_2^2 \phi_3 \phi_6 \phi_{14},
\frac{1}{2}q \phi_1^2 \phi_3 \phi_6 \phi_{7}, \\
\frac{-\sqrt{-7}}{7}q \phi_1^3 \phi_2^3 \phi_3 \phi_4 \phi_{6},  
\frac{-\sqrt{-7}}{7}q \phi_1^3 \phi_2^3 \phi_3 \phi_4 \phi_{6}, \\ 
\frac{-\sqrt{-7}}{7}q \phi_1^3 \phi_2^3 \phi_3 \phi_4 \phi_{6},  
q^3 \phi_7 \phi_{14}, \\
\frac{1}{2}q^4 \phi_2^3 \phi_4 \phi_6 \phi_{14},  
\frac{1}{2}q^4 \phi_2^3 \phi_4 \phi_6 \phi_{14}, \\
\frac{-1}{2}q^4 \phi_1^3\phi_3 \phi_4 \phi_{7},  
\frac{-1}{2}q^4 \phi_1^3\phi_3 \phi_4 \phi_{7}, \\ 
q^6 \phi_7 \phi_{14},  
\frac{\sqrt{-7}}{14}q^8 \phi_3 \phi_4 \phi_6 \phi_7' \phi_{14}'', \\
\frac{-\sqrt{-7}}{14}q^8 \phi_3 \phi_4 \phi_6 \phi_7'' \phi_{14}', 
\frac{1}{2}q^8 \phi_2^2 \phi_3 \phi_6 \phi_{14},  \\
\frac{1}{2}q^8 \phi_1^2 \phi_3 \phi_6 \phi_{7},  
\frac{\sqrt{-7}}{7}q^8 \phi_1^3 \phi_2^3 \phi_3 \phi_4 \phi_{6}, \\ 
\frac{\sqrt{-7}}{7}q^8 \phi_1^3 \phi_2^3 \phi_3 \phi_4 \phi_{6},  
\frac{\sqrt{-7}}{7}q^8 \phi_1^3 \phi_2^3 \phi_3 \phi_4 \phi_{6},  
q^{21}}$

 & $\makecell{0,0,0,1,2l+3,3l+10, \\ 3l+10,3l+10,0,3,3, 3l+6,3l+6, \\ 0,0,0,1,2l+3,3l+10,3l+10,3l+10,0}$                  \\ \hline

$\mathbf{B}_3(q)$              & $1$                                 & $1$                                              & $\makecell{\frac{1}{2}q^4\phi_4 \phi_6, q^3 \phi_3 \phi_6, \frac{1}{2}q^4 \phi_3 \phi_4, q^9, \frac{1}{2}q \phi_2^2 \phi_6, q^2 \phi_3 \phi_6, \\ \frac{1}{2} q \phi_3 \phi_4, \frac{1}{2}q^4\phi_2^2 \phi_6,1, \frac{1}{2}q\phi_4\phi_6,\frac{1}{2}q \phi_1^2\phi_3,\frac{1}{2}q^4\phi_1^2\phi_3}$ &    $0,0,0,0,1,0,0,1,0,0,2l+3,2l+3$   \\ \hline 

$\C_2(q)$              & $2x - 1$                         & $q^5 \phi_2 \phi_3 \phi_6$                       & $1,\frac{1}{2}q\phi_1^2, \frac{1}{2}q\phi_2^2,\frac{1}{2}q\phi_4,\frac{1}{2}q\phi_4,q^4$                                                                                                                            &   $1,2l+4,2,1,1,1$      \\ \hline

$\A_3(q)$              & $1$                         & $q^3 \phi_2 \phi_6$                       & $q^6,q^3\phi_3,q^2\phi_4,q\phi_3,1$                                                                                                                            &   $1,1,2,1,1$      \\ \hline

$\A_2(q)$              & $2x-2$                         & $q^6 \phi_2^2 \phi_4\phi_6$                       & $1, q\phi_2,q^3$                                                                                                                            &   $3,4,3$      \\ \hline

$\A_1(q)$              & $2x^2-5x+3$                         & $q^8 \phi_2^2 \phi_3\phi_4\phi_6$                       & $1, q$                                                                                                                            &   $3,3$      \\ \hline

$\A_1(q)+\A_1(q)$              & $x-1$ & $q^7 \phi_2 \phi_3\phi_4\phi_6$                       & $1, q,q,q^2$                                                                                                                            &   $2,2,2,2$      \\ \hline

$\A_0(q)$              & $\frac{4}{21} x^3-x^2+\frac{5}{3} x - \frac{6}{7} $& $q^9 \phi_2^3 \phi_3\phi_4\phi_6$                       & $1$ &   $4$      \\ \hline

$\A_1(q^2)$              & $x$& $q^7 \phi_1\phi_2^2 \phi_3\phi_6$                       & $1,q^2$ &   $l+4,l+4$      \\ 
\hline
$\A_1(q)$              & $x$& $q^8 \phi_1\phi_2 \phi_3\phi_4\phi_6$                       & $1,q$ &   $l+4,l+4$      \\ 
\hline

$\A_0(q)$              & $x^2-x$& $q^9 \phi_1\phi_2^2 \phi_3\phi_4\phi_6$                       & $1$ &   $l+5$     \\ 
\hline

\end{tabular}
\end{table}

\end{landscape}

\begin{landscape}
\begin{table}[]
\renewcommand{\arraystretch}{1.4}\tiny
\centering
\caption{Degrees of irreducible characters in $\Irr(\mathbf{X}(q))$, $q \equiv 3 \mod 4$ }
\label{table:spets2}
\begin{tabular}{|c|c|c|c|c|}
\hline
 type of $C_{\mathbf{X}(q)}(s)$ & $\#$ classes         & $|\mathbf{X}(q):C_{\mathbf{X}(q)}(s)|$                                     & $\lambda(1), \lambda \in \Irr^u(C_{\mathbf{X}(q)}(s))$                                                                                                                                                                                                                               & $v_2(\lambda(1)|\mathbf{X}(q):C_{\mathbf{X}(q)}(s)|)$ \\ \hline

\hline
$\mathbf{X}(q)$            & $1$         & $1$ & 

$\makecell{1,
\frac{\sqrt{-7}}{14}q \phi_3 \phi_4 \phi_6 \phi_7' \phi_{14}'', 
\frac{-\sqrt{-7}}{14}q \phi_3 \phi_4 \phi_6 \phi_7'' \phi_{14}', \\ 
\frac{1}{2}q \phi_2^2 \phi_3 \phi_6 \phi_{14},
\frac{1}{2}q \phi_1^2 \phi_3 \phi_6 \phi_{7}, \\
\frac{-\sqrt{-7}}{7}q \phi_1^3 \phi_2^3 \phi_3 \phi_4 \phi_{6},  
\frac{-\sqrt{-7}}{7}q \phi_1^3 \phi_2^3 \phi_3 \phi_4 \phi_{6}, \\ 
\frac{-\sqrt{-7}}{7}q \phi_1^3 \phi_2^3 \phi_3 \phi_4 \phi_{6},  
q^3 \phi_7 \phi_{14}, \\
\frac{1}{2}q^4 \phi_2^3 \phi_4 \phi_6 \phi_{14},  
\frac{1}{2}q^4 \phi_2^3 \phi_4 \phi_6 \phi_{14}, \\
\frac{-1}{2}q^4 \phi_1^3\phi_3 \phi_4 \phi_{7},  
\frac{-1}{2}q^4 \phi_1^3\phi_3 \phi_4 \phi_{7}, \\ 
q^6 \phi_7 \phi_{14},  
\frac{\sqrt{-7}}{14}q^8 \phi_3 \phi_4 \phi_6 \phi_7' \phi_{14}'', \\
\frac{-\sqrt{-7}}{14}q^8 \phi_3 \phi_4 \phi_6 \phi_7'' \phi_{14}', 
\frac{1}{2}q^8 \phi_2^2 \phi_3 \phi_6 \phi_{14},  \\
\frac{1}{2}q^8 \phi_1^2 \phi_3 \phi_6 \phi_{7},  
\frac{\sqrt{-7}}{7}q^8 \phi_1^3 \phi_2^3 \phi_3 \phi_4 \phi_{6}, \\ 
\frac{\sqrt{-7}}{7}q^8 \phi_1^3 \phi_2^3 \phi_3 \phi_4 \phi_{6},  
\frac{\sqrt{-7}}{7}q^8 \phi_1^3 \phi_2^3 \phi_3 \phi_4 \phi_{6},  
q^{21}}$

 & $\makecell{0,0,0,2l+3,1,3l+10, \\ 3l+10,3l+10,0,3l+6,3l+6, 3,3, \\ 0,0,0,2l+3,1,3l+10,3l+10,3l+10,0}$                  \\ \hline

$\mathbf{B}_3(q)$              & $1$                                 & $1$                                              & $\makecell{\frac{1}{2}q^4\phi_4 \phi_6, q^3 \phi_3 \phi_6, \frac{1}{2}q^4 \phi_3 \phi_4, q^9, \frac{1}{2}q \phi_2^2 \phi_6, q^2 \phi_3 \phi_6, \\ \frac{1}{2} q \phi_3 \phi_4, \frac{1}{2}q^4  \phi_2^2 \phi_6,1, \frac{1}{2}q\phi_4\phi_6,\frac{1}{2}q \phi_1^2\phi_3,\frac{1}{2}q^4\phi_1^2\phi_3}$ &    $0,0,0,0,2l+3,0,0,2l+3,0,0,1,1$   \\ \hline 

$\C_2(q)$              & $2x - 1$                         & $q^5 \phi_1 \phi_3 \phi_6$                       & $1,\frac{1}{2}q\phi_1^2, \frac{1}{2}q\phi_2^2,\frac{1}{2}q\phi_4,\frac{1}{2}q\phi_4,q^4$                                                                                                                            &   $1,2,2l+4,1,1,1$      \\ \hline

$^2\A_3(q)$              & $1$                         & $q^3 \phi_1 \phi_3$                       & $q^6,q^3\phi_6,q^2\phi_4,q\phi_6,1$                                                                                                                            &   $1,1,2,1,1$      \\ \hline

$^2\A_2(q)$              & $2x-2$                         & $q^6 \phi_1^2 \phi_3\phi_4$                       & $1, q\phi_1,q^3$                                                                                                                            &   $3,4,3$      \\ \hline

$\A_1(q)$              & $2x^2-5x+3$                         & $q^8 \phi_1^2 \phi_3\phi_4\phi_6$                       & $1, q$                                                                                                                            &   $3,3$      \\ \hline

$\A_1(q)+\A_1(q)$              & $x-1$ & $q^7 \phi_1 \phi_3\phi_4\phi_6$                       & $1, q,q,q^2$                                                                                                                            &   $2,2,2,2$      \\ \hline

$\A_0(q)$              & $\frac{4}{21} x^3-x^2+\frac{5}{3} x - \frac{6}{7} $& $q^9 \phi_1^3 \phi_3\phi_4\phi_6$                       & $1$ &   $4$      \\ \hline

$\A_1(q^2)$              & $x$& $q^7 \phi_1^2\phi_2 \phi_3\phi_6$                       & $1,q^2$ &   $l+4,l+4$      \\ 
\hline

$\A_1(q)$              & $x$& $q^8 \phi_1\phi_2 \phi_3\phi_4\phi_6$                       & $1,q$ &   $l+4,l+4$      \\ 
\hline

$\A_0(q)$              & $x^2-x$& $q^9 \phi_1^2\phi_2 \phi_3\phi_4\phi_6$                       & $1$ &   $l+5$     \\ 
\hline

\end{tabular}
\end{table}

\end{landscape}

\bibliographystyle{amsalpha}
\bibliography{mybib}
\end{document}